\documentclass[11pt,reqno]{amsart}

\usepackage[T1]{fontenc}
\usepackage[utf8]{inputenc}
\usepackage{amsmath,amssymb,amsthm,mathtools}
\usepackage[margin=1.15in]{geometry}
\usepackage{enumitem}
\usepackage[expansion=false]{microtype}
\usepackage{xcolor}
\usepackage[colorlinks=true,linkcolor=blue!50!black,citecolor=blue!50!black,
            urlcolor=blue!50!black]{hyperref}

\numberwithin{equation}{section}

\theoremstyle{plain}
\newtheorem{theorem}{Theorem}[section]
\newtheorem{lemma}[theorem]{Lemma}
\newtheorem{proposition}[theorem]{Proposition}
\newtheorem{corollary}[theorem]{Corollary}

\theoremstyle{definition}
\newtheorem{definition}[theorem]{Definition}
\newtheorem{example}[theorem]{Example}
\newtheorem{remark}[theorem]{Remark}

\newtheorem{innerclaim}{Claim}
\newenvironment{claim}{\begin{innerclaim}}{\end{innerclaim}}
\newenvironment{claimproof}
  {\par\smallskip\noindent\textit{Proof of the claim.}\ }
  {\hfill$\square$\par\smallskip}

\newcommand{\R}{\mathbb{R}}
\newcommand{\Lup}{L^{\mathrm{up}}}
\newcommand{\Ldown}{L^{\mathrm{down}}}
\newcommand{\eps}{\varepsilon}
\newcommand{\ones}{\mathbf{1}}
\newcommand{\T}{\mathsf{T}}
\newcommand{\Sset}{\mathcal S}
\newcommand{\Tset}{\mathcal T}
\newcommand{\Aset}{\mathcal A}
\newcommand{\Bset}{\mathcal B}
\newcommand{\Fset}{\mathcal F}
\DeclareMathOperator{\spn}{span}
\DeclareMathOperator{\sgn}{sgn}
\DeclareMathOperator{\adj}{adj}

\DeclareMathOperator{\tr}{tr}

\begin{document}

\title[Spectral bounds and shifted complexes]
{Spectral Bounds and Shifted Complexes: Eigenvalues of the Up-Laplacian via Face Degrees}

\author{Vinayak Gupta}
\address{Vinayak Gupta\\
Department of Mathematics, Indian Institute of Technology Guwahati,
Guwahati 781039, India}
\email{vinayakgupta1729v@gmail.com}


\begin{abstract}
Let $K$ be a finite $k$-dimensional simplicial complex with $k\ge1$, and let
$\Lup_{k-1}(K)$ be its $(k-1)$-dimensional up-Laplacian. We resolve the
equality case of a conjecture of Duval and Reiner: for a pure complex, the nonzero spectrum of
$\Lup_{k-1}(K)$ equals ${d^{\mathrm v}(K)}^{\T}$, the conjugate vertex-degree partition, if and
only if the complex is isomorphic to a shifted complex. In fact, we prove more: equality of the
second power sums alone already forces shiftedness, and the gap between them equals twice the
number of failed elementary shifts of facets.

We next characterize equality in the Duval--Reiner bound
$\lambda_1(\Lup_{k-1}(K))\ge d_1(K)+k$, where $d_1(K)$ is the maximum upper degree of a ridge.
We introduce the \emph{clean ridge}, a higher-dimensional analogue of a dominating vertex, and
combine it with a signed Sachs expansion for the associated signed facet-adjacency matrix to
obtain a necessary and sufficient combinatorial criterion. When the signed facet-adjacency
graph is connected and balanced---with balance equivalent to disorientability in dimension
$k$---the criterion reduces to the simple condition that all facets share a common ridge.

Finally, whenever $K$ has at least two facets, we prove the sharp bound
\[
  \lambda_2\bigl(\Lup_{k-1}(K)\bigr)\ge d_2(K)+k-1.
\]
This gives a higher-dimensional analogue of the second-index Brouwer--Haemers bound for graphs.
The natural termwise extension to higher eigenvalue indices fails already at $m=3$, as a small
explicit example shows.
\end{abstract}

\maketitle

\section{Introduction}

For a graph, many results compare the Laplacian eigenvalues with the vertex degrees.
Some give exact formulas for special classes of graphs; others bound individual
eigenvalues in terms of the ordered degree sequence. This paper develops
analogues of both types for simplicial complexes.

For a $k$-dimensional simplicial complex, the $(k-1)$-dimensional up-Laplacian plays
the role of the graph Laplacian; when $k=1$ it is exactly the ordinary Laplacian of a
graph. Two notions of degree appear naturally: the degree of a vertex, which counts
the facets containing it, and the upper degree of a ridge, which counts the facets
containing that ridge. We address three questions. First, when is the up-Laplacian
spectrum determined exactly by the vertex degrees? Second, when does equality hold in
the standard lower bound for the largest up-Laplacian eigenvalue? Third, can the
second largest eigenvalue be controlled by the second largest ridge degree?

\smallskip
\noindent\emph{Spectral equality with the conjugate degree partition.}
Let $K$ be a finite pure $k$-dimensional simplicial complex with $k\ge1$, let
$\Lup_{k-1}(K)$ be its unnormalized $(k-1)$-dimensional up-Laplacian, and write
$\lambda(K)=(\lambda_1(K),\lambda_2(K),\ldots)$ for the nonzero eigenvalues arranged
in nonincreasing order. Let $d^{\mathrm v}(K)=(a_1,\ldots,a_n)$ be the vertex degrees
arranged as $a_1\ge\cdots\ge a_n$, and let ${d^{\mathrm v}(K)}^{\T}$ denote its
conjugate partition, whose $i$th entry is the number of vertices of degree at least
$i$. Duval and Reiner~\cite{DuvalReiner} proved that every shifted complex satisfies
$\lambda(K)={d^{\mathrm v}(K)}^{\T}$, and Conjecture~1.2 of the same paper asserted
that, for every uniform family, $\lambda(K)$ is majorized by ${d^{\mathrm v}(K)}^{\T}$,
with equality if and only if the family is isomorphic to a shifted family. The majorization
half has recently been disproved in every dimension at least two~\cite{Huang,ZhangSongFan};
the equality characterization remained open. We settle it:
\[
  \lambda(K)={d^{\mathrm v}(K)}^{\T}
  \quad\Longleftrightarrow\quad
  K\text{ is isomorphic to a shifted complex}.
\]
For $k=1$ this recovers the classical result that graphs whose Laplacian spectrum
equals the conjugate degree partition are precisely threshold graphs, which are the
shifted graphs~\cite{Merris}. The equality characterization is a corollary of a sharper
identity: after ordering the vertices by nonincreasing degree,
\[
  \sum_{r\ge1}\bigl(d^{\mathrm v,\T}_r(K)\bigr)^2-\sum_r\lambda_r(K)^2
  \;=\;2\,\#\{\text{failed elementary shifts of facets}\},
\]
so equality of just the second power sums forces shiftedness. This defect formula,
proved in Section~\ref{sec:shifted}, is the engine behind the equality
characterization.

\smallskip
\noindent\emph{Equality in the largest-eigenvalue bound.}
Write $D=d_1(K)$ for the maximum upper degree of a ridge. A second result of Duval
and Reiner~\cite{DuvalReiner} gives $\lambda_1(\Lup_{k-1}(K))\ge D+k$, generalizing
$\lambda_1(L(G))\ge\Delta(G)+1$ of Grone and Merris~\cite{GroneMerris}. For connected
graphs, equality holds if and only if the graph has a dominating vertex. We identify
the corresponding local object in higher dimensions: the \emph{clean ridge}. The fan
$\Fset_\rho$ of a maximum-degree ridge $\rho$ is the set of facets containing it, each
obtained by adjoining one vertex to $\rho$. We call $\rho$ clean if every facet
outside its fan is adjacent to either zero or two facets in the fan. For a connected
graph, cleanliness of a maximum-degree vertex is exactly domination.

In dimensions $k\ge2$ cleanliness is necessary but not sufficient: a fan-vector
argument forces every maximum-degree ridge to be clean when
$\lambda_1(\Lup_{k-1}(K))=D+k$, but an explicit three-dimensional example
(Example~\ref{ex:dim3}) exhibits a complex whose unique maximum-degree ridge is clean
and yet $\lambda_1>D+k$. The remaining obstruction is global and is encoded by the signed
facet-adjacency graph. Expanding the relevant principal minors over elementary subgraphs
built from isolated vertices, edges, and signed cycles---a signed Sachs expansion---we obtain
a necessary and sufficient criterion (Theorem~\ref{thm:refined}). The criterion simplifies dramatically when
the signed facet-adjacency graph is connected and balanced, with balance equivalent
to disorientability in dimension $k$ in the sense of~\cite{EidiMukherjee}.
Disorientability is a higher-dimensional analogue of bipartiteness. In this setting,
equality holds if and only if all facets share a common ridge. For graphs this
recovers the characterization of stars.

\smallskip
\noindent\emph{A bound on the second largest eigenvalue.}
For graphs, Li and Pan~\cite{LiPan} proved $\lambda_2(L(G))\ge d_2(G)$, later
subsumed by the Brouwer--Haemers bound~\cite{BrouwerHaemers}. Replacing vertex
degrees by ordered ridge upper degrees, we prove the sharp inequality
\[
  \lambda_2\bigl(\Lup_{k-1}(K)\bigr)\ge d_2(K)+k-1
\]
whenever $K$ has at least two facets. The proof, given in Section~\ref{sec:second},
passes to the down Gram matrix associated with the boundary operator and compares two ridge
stars via a two-dimensional compression argument. For $k=1$ this is the Li--Pan inequality. The
natural termwise extension
\[
  \lambda_m\bigl(\Lup_{k-1}(K)\bigr)\ge d_m(K)-m+k+1
\]
fails already at $m=3$: the cone over a triangle has $d_3(K)=2$ but
$\lambda_3(\Lup_1(K))=1$. Thus the second-eigenvalue bound does not open a full
higher-dimensional Brouwer--Haemers family.

\section{Preliminaries}

Let $V$ be a finite set, whose elements are called \emph{vertices}. A \emph{simplicial
complex} $K$ on $V$ is a collection of subsets of $V$ such that, whenever $F\in K$ and
$\sigma\subseteq F$, one also has $\sigma\in K$. The elements of $K$ are called \emph{faces}.
Thus a face is simply a set of vertices, and the closure condition says that every subset of a
face is again a face.

A face containing $m+1$ vertices has dimension $m$, and $K_j$ denotes the set of all
$j$-dimensional faces. The dimension of $K$ is the largest dimension of one of its faces. In
this paper $K$ is $k$-dimensional, where $k\ge1$, and its $k$-faces will be called
\emph{facets}. A
$(k-1)$-face is called a \emph{ridge}; it is a codimension-one face of each facet that contains
it. Two facets share a ridge when their intersection is a $(k-1)$-face.

For a ridge $\alpha$, its \emph{facet star} is the collection
$\{F\in K_k:\alpha\subset F\}$. Its cardinality
\[
  d(\alpha)=\bigl|\{F\in K_k:\alpha\subset F\}\bigr|
\]
is the \emph{upper degree} of $\alpha$. When $k=1$, the vertices and edges of $K$ are exactly
the vertices and edges of a graph; ridges are vertices, facets are edges, and the upper degree
is the usual vertex degree.

An orientation of a face is an ordering of its vertices, where two orderings are identified if
they differ by an even permutation. Choose one orientation for each face, and let
$C_j(K;\R)$ be the real vector space having the chosen oriented $j$-faces as a basis; reversing
an orientation changes the sign of the corresponding basis vector. The simplicial boundary operator
\[
  \partial_k:C_k(K;\R)\longrightarrow C_{k-1}(K;\R)
\]
is defined on an oriented $k$-face by
\[
  \partial_k[v_0,\dots,v_k]
  =\sum_{i=0}^{k}{(-1)}^{i}[v_0,\dots,\widehat{v_i},\dots,v_k] .
\]
The matrix of $\partial_k$ with respect to these oriented-face bases is the boundary matrix
$B_k=B_k(K)$. Thus
\[
  B_k={\bigl([\sigma:F]\bigr)}_{\sigma\in K_{k-1},\,F\in K_k},
\]
where $[\sigma:F]$ denotes the entry in row $\sigma$ and column $F$. It is equal to $0$ when
$\sigma\not\subset F$, and it is
equal to $1$ or $-1$ when $\sigma$ occurs in the oriented boundary of $F$. Explicitly, if
$F=[v_0,\dots,v_k]$ and $\sigma=F\setminus\{v_i\}$, then $[\sigma:F]={(-1)}^{i}\sgn(\pi)$, where
$\pi$ is the permutation carrying $(v_0,\dots,\widehat{v_i},\dots,v_k)$ to the chosen ordering
of $\sigma$. We record for later use that this description makes sense for \emph{any} pair
$\sigma\subset F$ of finite sets with $|F|=|\sigma|+1$, whether or not they are faces of $K$;
this will matter in Lemma~\ref{lem:sign}, which concerns an abstract simplex that is never a
face of a $k$-dimensional complex.

The $(k-1)$-dimensional up-Laplacian is
\[
  \Lup_{k-1}(K)=B_k B_k^{\T} .
\]
For every real symmetric matrix $A$, we write its eigenvalues, counted with multiplicity, in
nonincreasing order as $\lambda_1(A)\ge\lambda_2(A)\ge\cdots$. The numbers $d_i(K)$ are the upper degrees
of the ridges arranged in nonincreasing order. Since $K_k$ is
nonempty and every $k$-face has $k+1\ge2$ faces of dimension $k-1$, the set $K_{k-1}$ has at
least two elements and $d_2(K)$ is defined.

The (unnormalized) $k$-dimensional down-Laplacian is
\[
  \Ldown_k(K)=B_k^{\T}B_k .
\]
We shall write $M(K)=\Ldown_k(K)$. The matrices
$\Lup_{k-1}(K)=B_k B_k^{\T}$ and $\Ldown_k(K)=B_k^{\T}B_k$ have the same nonzero eigenvalues,
with the same multiplicities; see Lemma~\ref{lem:gram}. Thus their nonzero spectral information
is equivalent, although the multiplicity of the eigenvalue zero may differ. This allows us to
work with whichever of the two matrices is more convenient.

More generally, writing $B_j$ for the $j$th boundary matrix, the adjacent Laplacians
\[
  \Lup_j(K)=B_{j+1}B_{j+1}^{\T}
  \qquad\text{and}\qquad
  \Ldown_{j+1}(K)=B_{j+1}^{\T}B_{j+1}
\]
have the same nonzero spectrum. Moreover, the full $j$-dimensional Laplacian
$L_j(K)=\Lup_j(K)+\Ldown_j(K)$ has, as its nonzero spectrum, the multiset union of the
nonzero spectra of its up- and down-parts. Indeed, $B_jB_{j+1}=0$ implies that these two parts
annihilate each other. Thus, across all dimensions, the nonzero full, up-, and down-Laplacian
spectra determine one another. In this paper we use only the adjacent pair
$\Lup_{k-1}(K)$ and $\Ldown_k(K)$.

If distinct $k$-faces $F$ and $G$ share a $(k-1)$-face $R$, define their incidence product by
\begin{equation}\label{eq:epsdef}
  \eps(F,G)=[R:F][R:G] .
\end{equation}
This definition is independent of the orientation chosen for $R$, because reversing that
orientation changes both factors. It does change sign if the orientation of exactly one of
$F,G$ is reversed. Since two distinct simplices share at most one ridge, the entries of $M(K)$
satisfy
\[
  M_{F,F}=k+1 ,
\]
and, for $F\ne G$,
\[
  M_{F,G}=\eps(F,G)
\]
when $F$ and $G$ share a ridge, while $M_{F,G}=0$ otherwise. Indeed,
$M_{F,G}=\sum_{R\in K_{k-1}}[R:F][R:G]$; for $F=G$ exactly the $k+1$ ridges of $F$ contribute
$1$ each, and for $F\ne G$ at most one ridge can be contained in both.

Throughout, $J$ denotes an all-ones matrix and $\ones$ an all-ones vector, of the size
indicated by the context, and $A[\Aset]$ denotes the principal submatrix of $A$ on the index
set $\Aset$, while $A[\Aset,\Bset]$ denotes the rectangular block with rows $\Aset$ and columns
$\Bset$.

\begin{definition}[Signed facet-adjacency graph]\label{def:sigma}
The \emph{signed facet-adjacency graph} $\Sigma(K)$ has vertex set $K_k$; two distinct facets are
adjacent when they share a ridge, and the sign of that adjacency is the incidence
product~\eqref{eq:epsdef}. Its signed adjacency matrix $S=S(K)$ is the symmetric matrix indexed
by $K_k$ with $S_{F,F}=0$ and, for $F\ne G$, with $S_{F,G}=\eps(F,G)$ if $F$ and $G$ are adjacent
and $S_{F,G}=0$ otherwise. We write $\Gamma(K)$ for the underlying unsigned graph, and for
$\Fset\subseteq K_k$ we write $\Gamma[\Fset]$ for the induced unsigned graph and
\[
  d_{\Fset}=\max_{F\in\Fset}\deg_{\Gamma[\Fset]}(F)
\]
for its maximum degree.
\end{definition}

The description of $M$ above says precisely that
\begin{equation}\label{eq:MS}
  M(K)=B_k^{\T}B_k=(k+1)I+S(K),
\end{equation}
and hence, by Lemma~\ref{lem:gram},
\begin{equation}\label{eq:lamS}
  \lambda_1\bigl(\Lup_{k-1}(K)\bigr)=k+1+\lambda_1\bigl(S(K)\bigr).
\end{equation}
The parameter $d_{\Fset}$ of Definition~\ref{def:sigma} should not be confused with the upper
degrees $d_i(K)$: the former is a maximum degree in a facet-adjacency graph, the latter a
maximum number of facets through a ridge.

\begin{lemma}\label{lem:deletion}
Suppose that $H$ is obtained from $K$ by deleting some $k$-faces while retaining the same row
set for the boundary matrix. Then, for every index $i$,
\[
  \lambda_i\bigl(\Lup_{k-1}(K)\bigr)\ \ge\ \lambda_i\bigl(\Lup_{k-1}(H)\bigr) .
\]
\end{lemma}

\begin{proof}
After ordering the columns of $B_k(K)$, write
\[
  B_k(K)=\begin{pmatrix}B_k(H) & C\end{pmatrix},
\]
where the columns of $C$ correspond to the deleted $k$-faces. Consequently,
\[
  \Lup_{k-1}(K)=B_k(K){B_k(K)}^{\T}
  =B_k(H){B_k(H)}^{\T}+CC^{\T}
  =\Lup_{k-1}(H)+CC^{\T}.
\]
The matrix $CC^{\T}$ is positive semidefinite. Weyl monotonicity therefore gives the required
inequalities: if $A\succeq B$ then $z^{\T}Az\ge z^{\T}Bz$ for every $z$, so the maximum over
$i$-dimensional subspaces of the minimal Rayleigh quotient is at least as large for $A$ as for
$B$.
\end{proof}

The following standard fact is an immediate consequence of the singular-value decomposition;
we record it for later use.

\begin{lemma}\label{lem:gram}
Let $B$ be a real matrix.
\begin{enumerate}[label=\textup{(\alph*)},leftmargin=2.6em]
\item The matrices $BB^{\T}$ and $B^{\T}B$ have the same nonzero eigenvalues,
with the same multiplicities.
\item\label{it:gram-b} If $\lambda_2(B^{\T}B)>0$, then $\lambda_2(BB^{\T})=\lambda_2(B^{\T}B)$.
\end{enumerate}
\end{lemma}

\begin{lemma}\label{lem:compress}
Let $A$ be a real symmetric $n\times n$ matrix and let $X\in\R^{n\times r}$ satisfy
$X^{\T}X=I_r$. Then
\[
  \lambda_i(A)\ \ge\ \lambda_i\bigl(X^{\T}AX\bigr),\qquad 1\le i\le r .
\]
In particular:
\begin{enumerate}[label=\textup{(\alph*)},leftmargin=2.6em]
\item\label{it:comp-a} if there is a two-dimensional subspace $U$ with
$z^{\T}Az/z^{\T}z\ge\alpha$ for every nonzero $z\in U$, then $\lambda_2(A)\ge\alpha$;
\item\label{it:comp-b} if $\Aset$ is a set of indices, then
$\lambda_i(A)\ge\lambda_i(A[\Aset])$ for $1\le i\le|\Aset|$.
\end{enumerate}
\end{lemma}

\begin{proof}
Fix $i\le r$. By the Courant--Fischer theorem applied to $X^{\T}AX$, there is an
$i$-dimensional subspace $W\subseteq\R^{r}$ with
\[
  \min_{0\ne u\in W}\frac{u^{\T}(X^{\T}AX)u}{u^{\T}u}=\lambda_i\bigl(X^{\T}AX\bigr).
\]
Since $X^{\T}X=I_r$, the map $u\mapsto Xu$ is injective, so $XW$ is an $i$-dimensional subspace
of $\R^{n}$; and for $z=Xu$ we have $z^{\T}Az=u^{\T}(X^{\T}AX)u$ and
$z^{\T}z=u^{\T}X^{\T}Xu=u^{\T}u$. Hence the minimal Rayleigh quotient of $A$ on $XW$ equals
$\lambda_i(X^{\T}AX)$, and Courant--Fischer applied to $A$,
\[
  \lambda_i(A)=\max_{\dim W'=i}\ \min_{0\ne z\in W'}\frac{z^{\T}Az}{z^{\T}z},
\]
gives $\lambda_i(A)\ge\lambda_i(X^{\T}AX)$.

For~\ref{it:comp-a}, take $r=2$ and let the columns of $X$ be an orthonormal basis of $U$. Then
$\lambda_2(X^{\T}AX)$ is the smaller of the two eigenvalues of the $2\times2$ matrix
$X^{\T}AX$, which by the displayed identity equals the minimum of the Rayleigh quotient of $A$
over $U$, hence is at least $\alpha$. For~\ref{it:comp-b}, take the columns of $X$ to be the
standard basis vectors indexed by $\Aset$, so that $X^{\T}AX=A[\Aset]$; this is the Cauchy
interlacing inequality, cf.~\cite{HornJohnson}.
\end{proof}

The next lemma shows that facet orientations can be chosen so that all incidence products
within a facet star are positive, without changing the spectrum of the down-Laplacian. We shall
use its simultaneous version for two facet stars.

\begin{lemma}\label{lem:switch}
Let $\alpha$ be a ridge and let $\Aset$ be a nonempty collection of facets containing
$\alpha$. The facets can be reoriented so that
\[
  \eps(F,F')=1\qquad\text{for all distinct }F,F'\in\Aset.
\]
Equivalently, there is a diagonal matrix $E$, indexed by the facets and having diagonal entries
in $\{1,-1\}$, such that the down-Laplacian matrix after this reorientation is $EME$ and
\[
  (EME)[\Aset]=kI+J .
\]
Here $E_{F,F}=-1$ means that the orientation of $F$ is reversed, whereas $E_{F,F}=1$ means that
it is left unchanged. Moreover, $EME$ has the same eigenvalues as $M$.

Moreover, let $\beta$ be another ridge and let $\Bset$ be a collection of facets containing
$\beta$. If $|\Aset\cap\Bset|\le1$, then the orientations can be chosen so that both
$(EME)[\Aset]=kI+J$ and $(EME)[\Bset]=kI+J$ hold simultaneously.
\end{lemma}

\begin{proof}
Any two distinct facets $F,F'\in\Aset$ intersect precisely in $\alpha$. Indeed, both contain
$\alpha$, and since they are distinct sets of $k+1$ vertices, their intersection has at most
$k=|\alpha|$ vertices. Thus they share the ridge $\alpha$.

Choose a reference facet $F_*\in\Aset$. For every $F\in\Aset$, set
\[
  E_{F,F}=[\alpha:F][\alpha:F_*] ,
\]
and set $E_{F,F}=1$ for every facet outside $\Aset$. Let $B_k'=B_k E$ be the boundary matrix
after the corresponding reorientation. For each $F\in\Aset$, its new incidence with $\alpha$
is
\[
  [\alpha:F]'=[\alpha:F]E_{F,F}
  ={[\alpha:F]}^{2}[\alpha:F_*]=[\alpha:F_*].
\]
Hence all facets in $\Aset$ have the same incidence sign with $\alpha$, and therefore
$\eps'(F,F')=[\alpha:F]'[\alpha:F']'=1$ for distinct $F,F'\in\Aset$. The diagonal entries of
the down-Laplacian remain $k+1$, so its principal block on $\Aset$ is $kI+J$.

Moreover,
\[
  {(B_k')}^{\T}B_k'={(B_k E)}^{\T}(B_k E)=EME .
\]
Thus $EME={(B_k E)}^{\T}(B_k E)$ is the down-Laplacian matrix corresponding to the new choice of
facet orientations. Since $E^{-1}=E=E^{\T}$, it is orthogonally similar to $M$, so no
eigenvalue changes. Moreover, its entries are described by the incidence products recorded
before Lemma~\ref{lem:deletion}, now computed with respect to the new orientations.

Finally, suppose $\Aset$ and $\Bset$ are given with $|\Aset\cap\Bset|\le1$. If
$\Aset\cap\Bset=\varnothing$, define the diagonal entries of $E$ separately on the two index
sets by the recipe above, with reference faces $F_*\in\Aset$ and $G_*\in\Bset$; there is no
conflict. If $\Aset\cap\Bset=\{F_0\}$, use $F_0$ as the reference face for both stars. The two
definitions assign to $F_0$ the values ${[\alpha:F_0]}^{2}=1$ and ${[\beta:F_0]}^{2}=1$ respectively,
hence agree, and they are the only possible source of conflict. In both situations the
computation above applies verbatim to each of $\Aset$ and $\Bset$.
\end{proof}

\begin{lemma}\label{lem:triangle}
Let $A,B,C$ be three distinct $k$-faces that pairwise share ridges. Then either they all contain
one common ridge, or they are three facets of the same abstract $(k+1)$-simplex.
\end{lemma}

\begin{proof}
Write $A=R\cup\{a\}$ and $B=R\cup\{b\}$, where $R=A\cap B$ is their common ridge. If
$R\subset C$, the first alternative holds. Otherwise, since $C$ shares a ridge with $A$, there
are $r\in R$ and a vertex $x\notin A$ such that
\[
  C=(A\setminus\{r\})\cup\{x\}=(R\setminus\{r\})\cup\{a,x\}.
\]
The intersection of this set with $B=R\cup\{b\}$ already contains the $k-1$ vertices of
$R\setminus\{r\}$. For it to be a ridge, one more vertex must lie in $B$, and therefore
$x=b$. Hence $A,B,C$ are obtained from the $(k+2)$-element set $R\cup\{a,b\}$ by deleting
$b,a,r$, respectively, so they are facets of one abstract $(k+1)$-simplex.
\end{proof}

\begin{lemma}\label{lem:sign}
Let $A,B,C$ be three distinct $k$-faces which are facets of the same abstract
$(k+1)$-simplex. Then, for arbitrary orientations of $A,B,C$,
\[
  \eps(A,B)\,\eps(B,C)\,\eps(C,A)=-1 .
\]
\end{lemma}

\begin{proof}
Write the vertices of the abstract $(k+1)$-simplex in the order
\[
  [w_0,w_1,\dots,w_{k+1}] .
\]
For each $i$, let
\[
  \Delta_i=[w_0,\dots,\widehat{w_i},\dots,w_{k+1}]
\]
be the facet obtained by deleting $w_i$, oriented by the displayed order of its remaining
vertices. Any two of these facets meet in $k$ vertices, since
$\Delta_i\cap\Delta_j=W\setminus\{w_i,w_j\}$; hence all three incidence products in the
statement are defined.

Take $i<j$, and orient the common ridge $\sigma_{ij}$ by the remaining vertex order after
deleting $w_i$ and $w_j$. In $\Delta_i$ the vertex $w_j$ occupies position $j-1$, because the
deletion of $w_i$ has shifted it down by one place, and deleting it leaves the vertices of
$\sigma_{ij}$ in the chosen order; hence
\[
  [\sigma_{ij}:\Delta_i]={(-1)}^{j-1}.
\]
In $\Delta_j$ the vertex $w_i$ occupies position $i$, since the deletion of $w_j$ does not
affect the places of the vertices preceding it, and hence
\[
  [\sigma_{ij}:\Delta_j]={(-1)}^{i}.
\]
It follows that
\[
  \eps(\Delta_i,\Delta_j)={(-1)}^{i+j-1}.
\]
Now take $i<j<\ell$. Multiplying the three pairwise signs, and using the symmetry
$\eps(\Delta_\ell,\Delta_i)=\eps(\Delta_i,\Delta_\ell)$, gives
\[
  \eps(\Delta_i,\Delta_j)\,\eps(\Delta_j,\Delta_\ell)\,\eps(\Delta_\ell,\Delta_i)
  ={(-1)}^{(i+j-1)+(j+\ell-1)+(i+\ell-1)}={(-1)}^{2i+2j+2\ell-3}=-1 .
\]
This proves the identity for the displayed orientations. Reversing the orientation of one facet
changes the sign of exactly the two factors involving that facet, so the triple product does
not change. Since every choice of orientations of $A,B,C$ is obtained from the displayed one by
reversing a subset of them, the identity holds for arbitrary orientations.
\end{proof}

We record the following elementary characterization, which will be used repeatedly in the proof.

\begin{lemma}\label{lem:adj}
Let $\sigma\ne\tau$ be $(k-1)$-faces, put
\[
  \rho=\sigma\cap\tau,\qquad q=|\sigma\setminus\tau|=|\tau\setminus\sigma|\ge1,
  \qquad |\rho|=k-q ,
\]
and let $F=\sigma\cup\{u\}$ and $G=\tau\cup\{v\}$ be $k$-faces with $F\ne G$, where
$u\notin\sigma$ and $v\notin\tau$. Then
\[
  |F\cap G|=
  \begin{cases}
    k-q+1, & u=v,\\[2pt]
    k-q+|\sigma\cap\{v\}|+|\tau\cap\{u\}|, & u\ne v.
  \end{cases}
\]
Consequently $|F\cap G|\le k-q+2$, and:
\begin{enumerate}[label=\textup{(\alph*)},leftmargin=2.6em]
\item\label{it:adj-a} if $q\ge3$, then $F$ and $G$ never share a ridge;
\item\label{it:adj-b} if $q=1$, write $\sigma=\rho\cup\{a\}$ and $\tau=\rho\cup\{b\}$. If in
addition $u\ne b$ and $v\ne a$, then $F$ and $G$ share a ridge if and only if $u=v$, and in
that case $F\cap G=\rho\cup\{u\}$ and $u\notin\rho\cup\{a,b\}$;
\item\label{it:adj-c} if $q=2$, write $\sigma=\rho\cup\{a_1,a_2\}$ and
$\tau=\rho\cup\{b_1,b_2\}$. Then $F$ and $G$ share a ridge if and only if
\[
  F=F_i:=\sigma\cup\{b_i\}\quad\text{and}\quad G=G_j:=\tau\cup\{a_j\}
\]
for some $i,j\in\{1,2\}$. Moreover,
$W=\sigma\cup\tau=\rho\cup\{a_1,a_2,b_1,b_2\}$ is a $(k+2)$-element vertex set, not
necessarily a face of $K$, and $F_1,F_2,G_1,G_2$ are four pairwise distinct
$(k+1)$-element subsets of $W$. Furthermore,
$F_i\cap G_j=\rho\cup\{a_j,b_i\}$ is a ridge for \emph{every} pair $(i,j)$.
\end{enumerate}
\end{lemma}

\begin{proof}
Decompose
\[
  F\cap G=\rho\ \cup\ (\sigma\cap\{v\})\ \cup\ (\{u\}\cap\tau)\ \cup\ (\{u\}\cap\{v\}) .
\]
These four sets are pairwise disjoint. Indeed $\sigma\cap\{v\}\subseteq\sigma\setminus\rho$
because $v\notin\tau\supseteq\rho$, and likewise $\{u\}\cap\tau\subseteq\tau\setminus\rho$;
these two sets are disjoint from each other and from $\rho$. If $u=v$, then $u\notin\sigma$ and
$v\notin\tau$ force the second and third sets to be empty and the fourth to be $\{u\}$, whence
$|F\cap G|=|\rho|+1=k-q+1$. If $u\ne v$, the fourth set is empty and the second and third
contribute the two indicators. This gives the asserted expression for $|F\cap G|$, and hence the bound
$|F\cap G|\le k-q+2$ follows.

For~\ref{it:adj-a}, if $q\ge3$ then $|F\cap G|\le k-1<k$, so $F$ and $G$ never share a ridge.

For~\ref{it:adj-b}, let $q=1$, so $|\rho|=k-1$, and assume $u\ne b$ and $v\ne a$.
Since $\sigma=\rho\cup\{a\}$ and $\tau=\rho\cup\{b\}$, the conditions
$v\notin\tau$, $v\ne a$, $u\notin\sigma$, and $u\ne b$ imply that
$v\notin\sigma$ and $u\notin\tau$. Therefore, using the decomposition above,
\[
  F\cap G=
  \begin{cases}
    \rho\cup\{u\}, & u=v,\\
    \rho, & u\ne v.
  \end{cases}
\]
Hence $|F\cap G|=k$ if $u=v$, and $|F\cap G|=k-1$ otherwise.

For~\ref{it:adj-c}, let $q=2$, so $|\rho|=k-2$. We use the decomposition of $F\cap G$ given above.
If $u=v$, then $F\cap G=\rho\cup\{u\}$, which has $k-1$ elements, so $F$ and $G$ do not
share a ridge. Suppose now that $u\ne v$. Outside $\rho$, the decomposition shows that the only
possible elements of $F\cap G$ are $v$, provided that $v\in\sigma$, and $u$, provided that
$u\in\tau$. Thus $|F\cap G|=k$ holds precisely when $v\in\sigma$ and $u\in\tau$. Together with
$u\notin\sigma$ and $v\notin\tau$, this gives
$u\in\tau\setminus\sigma=\{b_1,b_2\}$ and
$v\in\sigma\setminus\tau=\{a_1,a_2\}$, that is, $F=F_i$ and $G=G_j$. Conversely, for any
$i,j\in\{1,2\}$,
\[
  F_i\cap G_j=\bigl(\rho\cup\{a_1,a_2,b_i\}\bigr)\cap\bigl(\rho\cup\{b_1,b_2,a_j\}\bigr)
  =\rho\cup\{a_j,b_i\},
\]
which has $(k-2)+2=k$ elements, so $F_i$ and $G_j$ do share a ridge. The four faces are
pairwise distinct: $F_1\ne F_2$ and $G_1\ne G_2$ are clear, and $F_i=G_j$ is impossible because
$F_i\supseteq\sigma$ while $\sigma\not\subseteq G_j$, as $a_{3-j}\notin G_j$. Finally
$|W|=(k-2)+4=k+2$, and each of the four faces is $W$ with one vertex removed.
\end{proof}

\section{Shifted complexes and a second-moment characterization}\label{sec:shifted}

Throughout this section, $K$ is a finite pure $k$-dimensional simplicial complex on a linearly
ordered vertex set, identified with $[n]=\{1,\dots,n\}$.

\begin{definition}[Componentwise order and shifted complex]
\label{def:componentwise}\label{def:shifted}
Let
\[
  F=\{f_1<\cdots<f_t\},\qquad G=\{g_1<\cdots<g_t\}
\]
be two subsets of $[n]$ of the same cardinality. Following Duval and
Reiner~\cite[Definition 2.1]{DuvalReiner}, we write $F\le_P G$ when
\[
  f_r\le g_r\qquad\text{for every }1\le r\le t.
\]
A simplicial complex $K$ on $[n]$ is \emph{shifted} if, whenever $G\in K$ and
$F\subseteq[n]$ satisfies $|F|=|G|$ and $F\le_P G$, one has $F\in K$. Equivalently, each
family of faces of fixed cardinality is shifted in the sense of Duval and Reiner.
\end{definition}

A simplicial complex $K$ is shifted precisely when it is closed under every elementary
replacement
\[
  \sigma\longmapsto(\sigma\setminus\{j\})\cup\{i\},
  \qquad i<j,\quad j\in\sigma,\quad i\notin\sigma.
\]
Indeed, the componentwise order is generated by such replacements. If $K$ is pure, it is enough
to check this condition on the facets: every face lies in a facet, and the required replacement
either remains inside that facet or follows from the corresponding replacement of the facet.

Here degrees count facets through vertices, rather than facets through ridges.

\begin{definition}[Vertex-degree and conjugate partitions]
\label{def:vertexdegree}\label{def:conjugate}
For a vertex $v\in[n]$, define its degree with respect to the $k$-faces of $K$ by
\[
  \deg_K(v)=\bigl|\{\sigma\in K_k:v\in\sigma\}\bigr|.
\]
This is the vertex degree in the $(k+1)$-uniform family $K_k$ used by Duval and
Reiner~\cite[Section 2]{DuvalReiner}.

After relabelling the vertices if necessary, write
\[
  a_1\ge a_2\ge\cdots\ge a_n,
  \qquad a_i=\deg_K(i).
\]
The partition
\[
  d^{\mathrm v}(K)=(a_1,a_2,\ldots,a_n)
\]
is the \emph{vertex-degree partition} of $K$. The superscript distinguishes it from the ordered
ridge upper degrees $d_1(K),d_2(K),\dots$.

Its conjugate partition ${d^{\mathrm v}(K)}^{\T}$ is defined by
\[
  d^{\mathrm v,\T}_r(K)=\bigl|\{i\in[n]:a_i\ge r\}\bigr|,\qquad r\ge1.
\]
Thus $d^{\mathrm v,\T}_r(K)$ is the number of vertices of degree at least $r$. This is the
conjugate degree partition used by Duval and Reiner~\cite[Section 2]{DuvalReiner}.
\end{definition}

A shifted ordering is automatically a nonincreasing degree ordering. Indeed, if $i<j$, every
facet containing $j$ but not $i$ is mapped injectively to a facet containing $i$ but not $j$ by
\[
  \sigma\longmapsto(\sigma\setminus\{j\})\cup\{i\}.
\]
Hence $\deg_K(i)\ge\deg_K(j)$.

Recall that $B_k=B_k(K)$ is the oriented boundary matrix,
$\Lup_{k-1}(K)=B_kB_k^{\T}$, and $M(K)=B_k^{\T}B_k$. These matrices and their spectra are
independent of the chosen orientations up to orthogonal similarity.

\begin{definition}[Nonzero spectral partition]\label{def:spectrum}
Let
\[
  \lambda_1(K)\ge\lambda_2(K)\ge\cdots>0
\]
be the nonzero eigenvalues of $\Lup_{k-1}(K)$, counted with multiplicity, and write
\[
  \lambda(K)=(\lambda_1(K),\lambda_2(K),\ldots).
\]
We regard $\lambda(K)$ as a partition and suppress trailing zeroes.
\end{definition}

Duval and Reiner formulate their theorem for uniform set families. Applied to the
$(k+1)$-uniform family $K_k$, it takes the following form.

\begin{theorem}[Duval--Reiner~\cite{DuvalReiner}]\label{thm:DR}
If $K$ is a shifted pure $k$-dimensional simplicial complex, then
\[
  \lambda(K)={d^{\mathrm v}(K)}^{\T}.
\]
\end{theorem}

Relabel the vertices so that
\[
  a_1\ge a_2\ge\cdots\ge a_n,
  \qquad a_i=\deg_K(i).
\]
For $i<j$, we compare the admissible shifts from $j$ to $i$ with those that succeed.

\begin{definition}[Admissible and successful shifts]\label{def:NR}
For $i<j$, define
\[
  N_{ij}=\bigl|\{\sigma\in K_k:j\in\sigma,\ i\notin\sigma\}\bigr|
\]
and

\[
  R_{ij}=\bigl|\{\sigma\in K_k:j\in\sigma,\ i\notin\sigma,
  \ (\sigma\setminus\{j\})\cup\{i\}\in K_k\}\bigr|.
\]
Thus $N_{ij}$ counts admissible shifts, $R_{ij}$ counts successful shifts, and
$N_{ij}-R_{ij}$ counts failed shifts.
\end{definition}

\begin{lemma}\label{lem:NRshifted}
In the nonincreasing degree labeling,
\[
  K\text{ is shifted}
  \quad\text{if and only if}\quad
  N_{ij}=R_{ij}\text{ for every }i<j.
\]
\end{lemma}

\begin{proof}
This follows directly from Definition~\ref{def:NR}: the equality $N_{ij}=R_{ij}$ for every
$i<j$ says precisely that every admissible elementary shift of a facet succeeds, which is the
definition of shiftedness for a pure complex.
\end{proof}

\begin{proposition}\label{prop:degreemoment}
Put $m=f_k(K)$. Then
\[
  \sum_{r\ge1}\bigl(d^{\mathrm v,\T}_r(K)\bigr)^2
  =(k+1)^2m+2\sum_{i<j}N_{ij}.
\]
\end{proposition}

\begin{proof}
By Definition~\ref{def:conjugate},
\[
  d^{\mathrm v,\T}_r(K)=\bigl|\{i:a_i\ge r\}\bigr|.
\]
Therefore
\[
\begin{split}
\sum_{r\ge1}\bigl(d^{\mathrm v,\T}_r(K)\bigr)^2
&=\sum_{r\ge1}
  \bigl|\{(i,j)\in[n]^2:a_i\ge r,\ a_j\ge r\}\bigr|\\
&=\sum_{i=1}^n\sum_{j=1}^n
  \bigl|\{r\ge1:r\le\min(a_i,a_j)\}\bigr|\\
&=\sum_{i=1}^n\sum_{j=1}^n\min(a_i,a_j).
\end{split}
\]
Since $a_1\ge\cdots\ge a_n$, for $i<j$ we have $\min(a_i,a_j)=a_j$. Hence
\[
\sum_{r\ge1}\bigl(d^{\mathrm v,\T}_r(K)\bigr)^2
=\sum_i a_i+2\sum_{i<j}a_j.
\]
For $i<j$, put
\[
  c_{ij}=\bigl|\{\sigma\in K_k:i,j\in\sigma\}\bigr|.
\]
Thus $c_{ij}$ is the number of facets containing both $i$ and $j$.
Every $k$-face counted by $a_j$ either contains $i$ or does not. Thus
\[
a_j=c_{ij}+N_{ij}.
\]
Consequently,
\[
\sum_{r\ge1}\bigl(d^{\mathrm v,\T}_r(K)\bigr)^2
=\sum_i a_i+2\sum_{i<j}c_{ij}+2\sum_{i<j}N_{ij}.
\]
Each $k$-face has $k+1$ vertices, so
\[
\sum_i a_i=(k+1)m.
\]
Each $k$-face contains exactly $\binom{k+1}{2}$ unordered pairs of vertices, so
\[
\sum_{i<j}c_{ij}=\binom{k+1}{2}m.
\]
Therefore
\[
\begin{split}
\sum_{r\ge1}\bigl(d^{\mathrm v,\T}_r(K)\bigr)^2
&=(k+1)m+2\binom{k+1}{2}m+2\sum_{i<j}N_{ij}\\
&=(k+1)^2m+2\sum_{i<j}N_{ij}.
\end{split}
\]
\end{proof}

\begin{proposition}\label{prop:spectralmoment}
Put $m=f_k(K)$. Then
\[
  \sum_r \lambda_r(K)^2
  =(k+1)^2m+2\sum_{i<j}R_{ij}.
\]
\end{proposition}

\begin{proof}
Let $M=M(K)=B_k^{\T}B_k$. Its nonzero eigenvalues are the entries of $\lambda(K)$, so
\[
\sum_r \lambda_r(K)^2
=\tr(M^2)
=\sum_{\sigma,\tau\in K_k}M_{\sigma,\tau}^2.
\]
Every diagonal term contributes $(k+1)^2$. Let $A(K)$ be the number of unordered pairs of
distinct facets sharing a ridge. Every such pair contributes twice to the ordered sum, and hence
\[
\tr(M^2)=(k+1)^2m+2A(K).
\]
It remains to identify $A(K)$. Let $\{\sigma,\tau\}$ be an unordered pair of distinct facets
sharing a ridge. Then
\[
\sigma\triangle\tau=\{i,j\}
\]
for a unique pair of vertices $i<j$. Exactly one of $\sigma,\tau$ contains $j$ and not $i$; call it $\eta$. The other face is
\[
(\eta\setminus\{j\})\cup\{i\}.
\]
Thus the pair determines exactly one successful elementary shift counted by $R_{ij}$.

Conversely, every successful shift counted by some $R_{ij}$ produces an unordered pair of
distinct facets that differ in one vertex and hence share a ridge. The correspondence is
bijective, so
\[
A(K)=\sum_{i<j}R_{ij}.
\]
Substituting this into the trace formula gives
\[
\sum_r \lambda_r(K)^2
=(k+1)^2m+2\sum_{i<j}R_{ij}.
\]
\end{proof}

The two preceding formulas give the central identity of this section.

\begin{theorem}\label{thm:defect}
After the vertices are ordered by nonincreasing degree,
\[
\sum_{r\ge1}\bigl(d^{\mathrm v,\T}_r(K)\bigr)^2
-\sum_r \lambda_r(K)^2
=2\sum_{i<j}(N_{ij}-R_{ij}).
\]
In particular, the second-moment gap is twice the total number of failed elementary shifts of
facets.
\end{theorem}

\begin{proof}
This follows immediately from Propositions~\ref{prop:degreemoment}
and~\ref{prop:spectralmoment}, together with Definition~\ref{def:NR}.
\end{proof}

The defect formula immediately gives an inequality and its complete equality characterization.

\begin{theorem}\label{thm:momentcharacterization}
For every finite pure $k$-dimensional simplicial complex $K$,
\[
\sum_r \lambda_r(K)^2
\le
\sum_{r\ge1}\bigl(d^{\mathrm v,\T}_r(K)\bigr)^2.
\]
Moreover, equality holds if and only if $K$ is isomorphic to a shifted simplicial complex.

More precisely, if equality holds and the vertices are labeled so that
\[
\deg_K(1)\ge\deg_K(2)\ge\cdots\ge\deg_K(n),
\]
then $K$ is shifted in that labeling.
\end{theorem}

\begin{proof}
Choose a nonincreasing degree labeling. By Theorem~\ref{thm:defect}, the difference between the
right- and left-hand sides is $2\sum_{i<j}(N_{ij}-R_{ij})$, whose summands are nonnegative.
This proves the inequality. Equality holds precisely when $N_{ij}=R_{ij}$ for every $i<j$,
which, by Lemma~\ref{lem:NRshifted}, is equivalent to $K$ being shifted in this labeling.
Conversely, a shifted labeling is nonincreasing in degree, so the same lemma and
Theorem~\ref{thm:defect} give equality.
\end{proof}

\begin{remark}
Define
\[
\operatorname{shdef}(K)
=\sum_{i<j}(N_{ij}-R_{ij}),
\]
after arranging the vertices by nonincreasing degree. Then Theorem~\ref{thm:defect} says
\[
\operatorname{shdef}(K)
=\frac12\left(
\sum_{r\ge1}\bigl(d^{\mathrm v,\T}_r(K)\bigr)^2
-\sum_r \lambda_r(K)^2
\right).
\]
Thus the second-moment gap has an exact combinatorial interpretation: one half of it is the
number of failed elementary shifts.
\end{remark}

We can now characterize equality in the Duval--Reiner spectral identity.

\begin{theorem}[Exact Laplacian--degree equality]\label{thm:exactshifted}
Let $K$ be a finite pure $k$-dimensional simplicial complex. Then
\[
\lambda(K)={d^{\mathrm v}(K)}^{\T}
\]
if and only if $K$ is isomorphic to a shifted simplicial complex.

Moreover, if equality holds, then every nonincreasing degree labeling is shifted.
\end{theorem}

\begin{proof}
Assume first that
\[
\lambda(K)={d^{\mathrm v}(K)}^{\T}.
\]
Then the two sequences have the same second power sum:
\[
\sum_r \lambda_r(K)^2
=
\sum_{r\ge1}\bigl(d^{\mathrm v,\T}_r(K)\bigr)^2.
\]
Theorem~\ref{thm:momentcharacterization} implies that $K$ is shifted after the vertices are
ordered by nonincreasing degree.

Conversely, if $K$ is isomorphic to a shifted complex, choose a shifted labeling. The
Duval--Reiner theorem, Theorem~\ref{thm:DR}, gives
\[
\lambda(K)={d^{\mathrm v}(K)}^{\T}.
\]
\end{proof}

\begin{corollary}\label{cor:weakerhypothesis}
A finite pure $k$-dimensional simplicial complex $K$ is isomorphic to a shifted complex if and
only if
\[
  \sum_r \lambda_r(K)^2
  =\sum_{r\ge1}\bigl(d^{\mathrm v,\T}_r(K)\bigr)^2.
\]
\end{corollary}

\begin{proof}
This is exactly the equality statement in Theorem~\ref{thm:momentcharacterization}.
\end{proof}

\section{The largest up-Laplacian eigenvalue}\label{sec:largest}

Throughout this section we abbreviate
\[
  D=d_1(K)=\max_{\rho\in K_{k-1}}d(\rho),
\]
and we write $S=S(K)$ for the signed facet-adjacency matrix of Definition~\ref{def:sigma}. The
object of study is
\begin{equation}\label{eq:Qdef}
  Q=Q(K)=(D+k)I-B_k^{\T}B_k=(D-1)I-S ,
\end{equation}
the second equality being~\eqref{eq:MS}. The second form is independent of the dimension except
through the signed graph itself, and it is the form we use.

\subsection{The fan vector and the spectral reformulation}

\begin{definition}[Fan and fan vector]\label{def:fanvector}
Let $\rho$ be a ridge of degree $D$ and let
\[
  \Fset_\rho=\{F_1,\dots,F_D\},\qquad F_i=\rho\cup\{v_i\},
\]
be the set of facets containing $\rho$. We call $\Fset_\rho$ the \emph{fan} of $\rho$ and define
its \emph{fan vector} $x_\rho\in\R^{K_k}$ by
\begin{equation}\label{eq:fanvector}
  {(x_\rho)}_F=\begin{cases}[\rho:F], & \rho\subset F,\\ 0,&\text{otherwise.}\end{cases}
\end{equation}
\end{definition}

Building the incidence signs into the vector, rather than reorienting the fan, makes $x_\rho$
independent of any choice and lets us avoid switching in the statements below.

\begin{proposition}\label{prop:reform}
One always has $\lambda_1(S)\ge D-1$, and hence
$\lambda_1(\Lup_{k-1}(K))\ge D+k$. Moreover the following are equivalent:
\[
  \lambda_1\bigl(\Lup_{k-1}(K)\bigr)=D+k,\qquad
  \lambda_1(S)=D-1,\qquad
  Q\succeq0 .
\]
\end{proposition}

\begin{proof}
For distinct $F_i,F_j\in\Fset_\rho$ we have $F_i\cap F_j=\rho$ by the first paragraph of the
proof of Lemma~\ref{lem:switch}, so $S_{F_i,F_j}=[\rho:F_i][\rho:F_j]$. Therefore
\[
  x_\rho^{\T}Sx_\rho=\sum_{i\ne j}{[\rho:F_i]}^{2}{[\rho:F_j]}^{2}=D(D-1),\qquad
  x_\rho^{\T}x_\rho=D ,
\]
so the Rayleigh quotient of $x_\rho$ equals $D-1$ and $\lambda_1(S)\ge D-1$. Together
with~\eqref{eq:lamS} this gives the lower bound for the up-Laplacian, which is the inequality of
Duval and Reiner~\cite{DuvalReiner}.

Equation~\eqref{eq:lamS} shows that the first two equality statements are equivalent, and
by~\eqref{eq:Qdef}, $Q\succeq0$ is equivalent to $\lambda_1(S)\le D-1$. The lower bound just
proved turns this inequality into an equality.
\end{proof}

\subsection{Principal minors and signed Sachs subgraphs}

For $\Fset\subseteq K_k$ write $Q[\Fset]$ and $S[\Fset]$ for the corresponding principal
submatrices, and set
\[
  \Phi(\Fset)=\det Q[\Fset]=\det\bigl((D-1)I-S[\Fset]\bigr).
\]

\begin{theorem}\label{thm:allminors}
$\lambda_1(\Lup_{k-1}(K))=D+k$ if and only if $\Phi(\Fset)\ge0$ for every nonempty
$\Fset\subseteq K_k$ for which $\Gamma[\Fset]$ is connected.
\end{theorem}

\begin{proof}
By Proposition~\ref{prop:reform}, the equality is equivalent to $Q\succeq0$, and a real symmetric
matrix is positive semidefinite if and only if all of its principal minors are nonnegative.

It remains to justify the restriction to connected collections. If $\Gamma[\Fset]$ has connected
components $\Fset_1,\dots,\Fset_t$, then no signed adjacency joins different components, so after
a simultaneous permutation of rows and columns
\[
  Q[\Fset]=Q[\Fset_1]\oplus\cdots\oplus Q[\Fset_t],\qquad
  \Phi(\Fset)=\prod_{j=1}^{t}\Phi(\Fset_j).
\]
Thus nonnegativity for connected collections implies it for every collection.
\end{proof}

The determinant can be written entirely in terms of signed matchings and cycles, which makes its
combinatorial content explicit.
\begin{definition}[Spanning elementary subgraph]\label{def:sachs}
Let $H_{\Fset}$ be the signed graph induced by $\Sigma(K)$ on $\Fset$. A \emph{spanning
elementary subgraph} $R$ of $H_{\Fset}$ is a spanning subgraph whose connected components are
isolated vertices, isolated edges, or simple cycles of length at least three. Let $i(R)$ be the
number of isolated vertices, $m(R)$ the number of isolated-edge components and $c(R)$ the number
of cycle components. For a signed cycle $C$ put $\eps(C)=\prod_{e\in E(C)}\eps(e)$, and let
$\eps(R)$ be the product of the signs of the cycle components of $R$, the empty product being $1$.
\end{definition}

\begin{proposition}[Signed Sachs expansion]\label{prop:sachs}
For every nonempty $\Fset\subseteq K_k$,
\begin{equation}\label{eq:sachs}
  \Phi(\Fset)=\sum_{R}{(-1)}^{m(R)+c(R)}\,2^{\,c(R)}\,\eps(R)\,{(D-1)}^{i(R)},
\end{equation}
the sum being over all spanning elementary subgraphs $R$ of $H_{\Fset}$.
\end{proposition}

\begin{proof}
Expand $\det\bigl((D-1)I-S[\Fset]\bigr)$ by permutations. A fixed point contributes $D-1$. A
transposition supported on an edge $e$ contributes $-{\eps(e)}^{2}=-1$. A permutation cycle of length
at least three supported on a signed cycle $C$ contributes $-\eps(C)$, and the reverse traversal
of the same undirected cycle contributes the same quantity, so each cycle component contributes
$-2\eps(C)$. Multiplying over disjoint permutation cycles yields~\eqref{eq:sachs}.
\end{proof}

\subsection{Collections that are automatically harmless}

\begin{lemma}[Internal-degree reduction]\label{lem:gersh}
Let $\Fset\subseteq K_k$ be nonempty. If $d_{\Fset}\le D-1$, then $Q[\Fset]\succeq0$; in
particular $\Phi(\Fset)\ge0$.
\end{lemma}

\begin{proof}
The matrix $S[\Fset]$ is symmetric, has zero diagonal, and has entries of absolute value one
exactly at the adjacent pairs inside $\Fset$. Hence its $F$-th absolute row sum is the number of
neighbours of $F$ inside $\Fset$, which is at most $d_{\Fset}$. Gershgorin's theorem gives
$\lambda_1(S[\Fset])\le d_{\Fset}\le D-1$, so
$Q[\Fset]=(D-1)I-S[\Fset]\succeq0$.
\end{proof}

\begin{corollary}\label{cor:small}
If $|\Fset|\le D$, then $Q[\Fset]\succeq0$.
\end{corollary}

\begin{proof}
An induced graph on at most $D$ vertices has maximum degree at most $D-1$; apply
Lemma~\ref{lem:gersh}.
\end{proof}

Thus a negative principal minor can first occur at order $D+1$. Applying the same row-sum
argument to the full matrix yields a global upper bound. For graphs, the corresponding estimate
is the Anderson--Morley inequality; the following proposition gives its simplicial-complex
analogue.

\begin{proposition}[Anderson--Morley bound]\label{prop:AM}
For every finite $k$-dimensional simplicial complex $K$ one has
\[
  \lambda_1\bigl(\Lup_{k-1}(K)\bigr)\ \le\ \max_{F\in K_k}\ \sum_{\rho\subset F,\ \dim\rho=k-1} d(\rho).
\]
Consequently, if $\sum_{\rho\subset F}d(\rho)\le D+k$ for every $F\in K_k$, then
$\lambda_1(\Lup_{k-1}(K))=D+k$.
Here $D=d_1(K)$ denotes the largest upper degree of a ridge in $K$.
\end{proposition}

\begin{proof}
In the row of $B_k^{\T}B_k$ indexed by $F$ the diagonal entry is $k+1$, and for every ridge
$\rho\subset F$ there are $d(\rho)-1$ other facets sharing $\rho$ with $F$, each contributing an
off-diagonal entry of absolute value one. Distinct ridges give distinct facets, since a facet
sharing $\rho$ with $F$ determines $\rho$ as its intersection with $F$. Thus the absolute row sum
equals
\[
  (k+1)+\sum_{\rho\subset F}\bigl(d(\rho)-1\bigr)=\sum_{\rho\subset F}d(\rho),
\]
and the spectral radius of a symmetric matrix is at most its maximum absolute row sum. The final
assertion follows by combining this upper bound with Proposition~\ref{prop:reform}.
\end{proof}

For $k=1$, Proposition~\ref{prop:AM} is the classical Anderson--Morley inequality
\[
  \lambda_1(L(G))\le\max_{uv\in E(G)}\bigl(d(u)+d(v)\bigr).
\]

\subsection{Clean maximum-degree ridges}

Let $\rho$ be a ridge of degree $D$ with fan $\Fset_\rho$.

\begin{lemma}\label{lem:atmosttwo}
Every facet $U\notin\Fset_\rho$ is adjacent to at most two members of $\Fset_\rho$.
\end{lemma}

\begin{proof}
Suppose $U$ is adjacent to $F_i=\rho\cup\{v_i\}$, so $|U\cap F_i|=k$. If $v_i\notin U$ then
$U\cap F_i\subseteq\rho$, and $|U\cap F_i|=k=|\rho|$ would force $\rho\subseteq U$, contrary to
$U\notin\Fset_\rho$. Hence $v_i\in U$ and $|U\cap\rho|=k-1$, so
\[
  U=(\rho\setminus\{u\})\cup\{v_i,w\}
\]
for some $u\in\rho$ and some vertex $w$. Thus $U$ has exactly two vertices outside $\rho$, and at
most two of the fan vertices $v_1,\dots,v_D$ can occur among them.
\end{proof}

\begin{definition}[Clean ridge]\label{def:clean}
A ridge $\rho$ of degree $D$ is \emph{clean} if every facet outside $\Fset_\rho$ is adjacent to
either zero or two members of $\Fset_\rho$; equivalently, if no facet outside the fan is adjacent
to exactly one member of it.
\end{definition}

The following eigenvector formulation is the conceptual core of cleanliness.

\begin{proposition}\label{prop:fanvector}
Let $\rho$ be a ridge of degree $D$. Then $\rho$ is clean if and only if
$Sx_\rho=(D-1)x_\rho$.
Consequently, if $\lambda_1(\Lup_{k-1}(K))=D+k$, then every ridge of degree $D$ is clean.
\end{proposition}

\begin{proof}
Let $F_j\in\Fset_\rho$. Only the other fan members contribute to ${(Sx_\rho)}_{F_j}$, because
$x_\rho$ vanishes outside the fan, and each contributes
$\eps(F_j,F_i){(x_\rho)}_{F_i}=[\rho:F_j]{[\rho:F_i]}^{2}=[\rho:F_j]$. Hence
\[
  {(Sx_\rho)}_{F_j}=(D-1)[\rho:F_j]=(D-1){(x_\rho)}_{F_j},
\]
so the eigenvalue equation holds on the fan, whatever the orientations.

To complete the proof of the equivalence, let $U\notin\Fset_\rho$. Since ${(x_\rho)}_U=0$,
the eigenvalue equation at $U$ reads
\[
  \sum_{\substack{F_i\in\Fset_\rho\\ F_i\sim U}}\eps(U,F_i)\,[\rho:F_i]=0 .
\]
By Lemma~\ref{lem:atmosttwo}, the facet $U$ has zero, one, or two neighbours in the fan. With
no fan neighbour the sum is empty and therefore vanishes, whereas with exactly one fan neighbour
it equals $\pm1$ and does not vanish. If $U$ has two fan neighbours $F_i,F_j$, then
$F_i,F_j,U$ are pairwise adjacent but have no common ridge, because $F_i\cap F_j=\rho$ and
$U\notin\Fset_\rho$. Lemmas~\ref{lem:triangle} and~\ref{lem:sign} therefore give
\[
  \bigl(\eps(U,F_i)[\rho:F_i]\bigr)
  \bigl(\eps(U,F_j)[\rho:F_j]\bigr)
  =\eps(U,F_i)\eps(U,F_j)\eps(F_i,F_j)=-1 .
\]
Thus the two summands are opposite and cancel. Hence the eigenvalue equation holds outside the
fan if and only if no outside facet has exactly one fan neighbour, that is, if and only if
$\rho$ is clean.

For the last statement, suppose $\lambda_1(\Lup_{k-1}(K))=D+k$. Proposition~\ref{prop:reform}
gives $\lambda_1(S)=D-1$. The fan vector always has Rayleigh quotient $D-1$, so it attains
the maximum of the Rayleigh quotient and is therefore an eigenvector of $S$ for the eigenvalue
$D-1$. The equivalence just proved then gives cleanliness.
\end{proof}

\begin{remark}\label{rem:whynotsuff}
Proposition~\ref{prop:fanvector} shows that cleanliness of every ridge of degree $D$ is necessary
for the equality $\lambda_1(\Lup_{k-1}(K))=D+k$. It is not sufficient.
Example~\ref{ex:dim3} below gives an explicit counterexample. This is the reason for the
additional global condition in Theorem~\ref{thm:refined}.
\end{remark}

\subsection{The complete classification of minors of order \texorpdfstring{$D+1$}{D+1}}\label{sub:Dplus1}

The fan-vector proof gives necessity and the correct conceptual interpretation. To replace
\emph{every} principal minor of order $D+1$ by the clean-ridge condition one must also show that
no other configuration of $D+1$ facets produces a negative determinant. This subsection supplies
that converse.

\begin{lemma}[Bordered fan determinant]\label{lem:borderedfan}
Let $\rho$ be a ridge of degree $D$ and let $U\notin\Fset_\rho$. Then
\[
  \det Q[\Fset_\rho\cup\{U\}]
  =\begin{cases}
    -D^{D-2}, & \text{if $U$ has exactly one neighbour in $\Fset_\rho$},\\[2pt]
    0, & \text{if $U$ has zero or two neighbours in $\Fset_\rho$}.
  \end{cases}
\]
Consequently, $\rho$ is clean if and only if
\[
  Q[\Fset_\rho\cup\{U\}]\succeq0
  \qquad\text{for every }U\notin\Fset_\rho.
\]
\end{lemma}

\begin{proof}
Reversing facet orientations conjugates $Q$ by a diagonal sign matrix, and therefore leaves
every principal determinant unchanged. We may thus switch the fan coordinates by the diagonal
matrix with entries $[\rho:F_i]$ and leave the coordinate of $U$ unchanged. In these coordinates
the fan block of $Q$ is
\[
  A_D=(D-1)I_D-(J_D-I_D)=D I_D-J_D.
\]
This matrix has kernel $\spn\{\ones\}$ and eigenvalue $D$ with multiplicity $D-1$. Hence every
column of $\adj(A_D)$ is a multiple of $\ones$, and symmetry gives $\adj(A_D)=cJ_D$. Evaluating
a diagonal cofactor yields $c=D^{D-2}$, so
\begin{equation}\label{eq:adj}
  \adj(A_D)=D^{D-2}J_D .
\end{equation}

Let $b\in\R^D$ be the switched border column between $U$ and the fan. Its entries are
\[
  b_i=\begin{cases}
    -[\rho:F_i]\eps(F_i,U), & F_i\sim U,\\
    0, & F_i\not\sim U.
  \end{cases}
\]
Thus
\[
  Q[\Fset_\rho\cup\{U\}]\ \sim\ \begin{pmatrix}A_D & b\\ b^{\T} & D-1\end{pmatrix},
\]
where switching congruence preserves the determinant. Since $\det A_D=0$, the bordered
determinant formula and~\eqref{eq:adj} give
\begin{equation}\label{eq:bordered}
  \det Q[\Fset_\rho\cup\{U\}]=-b^{\T}\adj(A_D)b=-D^{D-2}{\bigl(\ones^{\T}b\bigr)}^{2} .
\end{equation}
If $U$ has no fan neighbour, then $b=0$. If it has exactly one fan neighbour, then
$\ones^{\T}b=\pm1$. If it has two fan neighbours, the cancellation in the proof of
Proposition~\ref{prop:fanvector} shows that the two nonzero entries of $b$ are opposite, and
hence $\ones^{\T}b=0$. Lemma~\ref{lem:atmosttwo} shows that these are all the possibilities,
and the determinant formula follows from~\eqref{eq:bordered}.

Every proper principal submatrix of $Q[\Fset_\rho\cup\{U\}]$ has order at most $D$ and is
positive semidefinite by Corollary~\ref{cor:small}. Hence the bordered matrix is positive
semidefinite if and only if its determinant is nonnegative. By the formula above, this occurs
if and only if $U$ has zero or two fan neighbours. Requiring this for every
$U\notin\Fset_\rho$ is precisely the definition of cleanliness.
\end{proof}

\begin{lemma}\label{lem:universal}
Let $\Fset$ contain exactly $D+1$ facets. If $Q[\Fset]$ is not positive semidefinite, then
$\Gamma[\Fset]$ has a universal vertex, that is, a facet adjacent to all the others in $\Fset$.
\end{lemma}

\begin{proof}
The fact that $Q[\Fset]=(D-1)I-S[\Fset]$ is not positive semidefinite means that
$\lambda_1(S[\Fset])>D-1$. Let $A_{\Fset}$ be the unsigned adjacency matrix of
$\Gamma[\Fset]$. Since $A_{\Fset}=|S[\Fset]|$ entrywise,
\[
  D-1<\lambda_1(S[\Fset])
  \le\lambda_1(A_{\Fset})
  \le d_{\Fset},
\]
where the last inequality is the maximum-degree bound. Hence $d_{\Fset}\ge D$. Since
$\Gamma[\Fset]$ has $D+1$ vertices, its maximum degree is at most $D$. Therefore
$d_{\Fset}=D$, so one facet is adjacent to all the others.
\end{proof}

Let $F_*=\{u_0,\dots,u_k\}$ be universal in a collection $\Fset$ of $D+1$ facets. Put
$\rho_j=F_*\setminus\{u_j\}$ for $0\le j\le k$, and let $\Aset_j$ be the collection of facets in
$\Fset\setminus\{F_*\}$ that contain $\rho_j$; every facet of $\Fset\setminus\{F_*\}$ lies in
exactly one $\Aset_j$, namely the one indexed by its intersection with $F_*$. Write
$r_j=|\Aset_j|$. Then
\begin{equation}\label{eq:groupsizes}
  \sum_{j=0}^{k}r_j=D,\qquad r_j\le D-1,
\end{equation}
the second inequality because $\rho_j$ is contained in $F_*$ and in every member of $\Aset_j$, so
that $d(\rho_j)\ge r_j+1$, while $d(\rho_j)\le D$.

Every $U\in\Aset_j$ has a unique representation $U=\rho_j\cup\{x_U\}$ with $x_U\notin F_*$.
Switch so that every signed edge incident with $F_*$ is positive, which is possible by
Lemma~\ref{lem:switch} applied to the star of $\rho_j$ for each $j$ in turn, or directly by
setting $E_{U,U}=\eps(F_*,U)$ for $U\ne F_*$. Write $\eps_0$ for the signs before switching and
$\eps'$ for the signs afterward. Then:

\begin{itemize}[leftmargin=1.6em,itemsep=0.25em]
\item two facets in the same group $\Aset_j$ are joined by a \emph{positive} edge. Indeed, for
$U,U'\in\Aset_j$, the original signs satisfy
\[
  \eps_0(U,U')=\eps_0(F_*,U)\eps_0(F_*,U'),
\]
since the product on the right is
${[\rho_j:F_*]}^{2}[\rho_j:U][\rho_j:U']$. Therefore the switched sign is
\[
  \eps'(U,U')
  =\eps_0(U,U')\eps_0(F_*,U)\eps_0(F_*,U')=1;
\]
\item two facets in different groups $\Aset_j,\Aset_l$ are adjacent exactly when they have the
same outside vertex. Indeed for $U=\rho_j\cup\{x\}$ and $U'=\rho_l\cup\{x'\}$ with $j\ne l$ and
$x,x'\notin F_*$ one has $|U\cap U'|=|\rho_j\cap\rho_l|+[x=x']=(k-1)+[x=x']$. When $x=x'$, the
three facets $F_*,U,U'$ are facets of the abstract $(k+1)$-simplex $F_*\cup\{x\}$, so
Lemma~\ref{lem:triangle} and Lemma~\ref{lem:sign} give
$\eps_0(F_*,U)\eps_0(U,U')\eps_0(U',F_*)=-1$. Hence
\[
  \eps'(U,U')
  =\eps_0(U,U')\eps_0(F_*,U)\eps_0(F_*,U')=-1,
\]
so the switched edge is \emph{negative}.
\end{itemize}

\begin{lemma}\label{lem:quadform}
Assume $D\ge3$. If $r_j\le D-2$ for every $j$, then $Q[\Fset]\succeq0$.
\end{lemma}

\begin{proof}
Let $y={(y_U)}_{U\in\Fset\setminus\{F_*\}}\in\R^{D}$ and put
\[
  s=\sum_U y_U,\qquad g_j=\sum_{U\in\Aset_j}y_U,\qquad h_x=\sum_{U:\,x_U=x}y_U .
\]
The diagonal contribution is $(D-1)(\alpha^2+\|y\|^2)$, and the positive edges from $F_*$
contribute $-2\alpha s$. The positive edges within the groups contribute
\[
  -2\sum_j\sum_{\{U,U'\}\subseteq\Aset_j}y_Uy_{U'}
  =-\sum_j g_j^2+\|y\|^2.
\]
The negative edges between different groups join precisely the pairs with the same outside
vertex, and therefore contribute
\[
  2\sum_x\sum_{\substack{\{U,U'\}:\,x_U=x_{U'}=x}}y_Uy_{U'}
  =\sum_x h_x^2-\|y\|^2.
\]
The groups $\Aset_j$ partition the coordinates of $y$, as do the classes determined by the
outside vertices. Thus the two occurrences of $\|y\|^2$ cancel, giving
\[
  \begin{pmatrix}\alpha\\ y\end{pmatrix}^{\T}Q[\Fset]\begin{pmatrix}\alpha\\ y\end{pmatrix}
  =(D-1)\alpha^2-2\alpha s+(D-1)\|y\|^2-\sum_{j}g_j^2+\sum_x h_x^2
\]
for every $\alpha\in\R$. Completing the square in $\alpha$ reduces the claim to
\[
  (D-1)\|y\|^2-\sum_j g_j^2+\sum_x h_x^2-\frac{s^2}{D-1}\ \ge\ 0 .
\]
By Cauchy--Schwarz and $r_j\le D-2$,
\[
  \sum_j g_j^2\le\sum_j r_j\sum_{U\in\Aset_j}y_U^2\le(D-2)\|y\|^2,
\]
so that $(D-1)\|y\|^2-\sum_j g_j^2\ge\|y\|^2$. There are $D$ coordinates in $y$, so
$\|y\|^2\ge s^2/D$; and there are at most $D$ distinct outside vertices with
$\sum_x h_x=s$, so another application of Cauchy--Schwarz gives $\sum_x h_x^2\ge s^2/D$.
Consequently
\[
  (D-1)\|y\|^2-\sum_j g_j^2+\sum_x h_x^2-\frac{s^2}{D-1}
  \ \ge\ \frac{s^2}{D}+\frac{s^2}{D}-\frac{s^2}{D-1}
  =\frac{D-2}{D(D-1)}\,s^2\ \ge\ 0 . \qedhere
\]
\end{proof}

\begin{theorem}[Classification of the first nontrivial minors]\label{thm:Dplus1}
Assume $D\ge2$ and let $\Fset$ be a connected collection of exactly $D+1$ facets. Then
$\Phi(\Fset)<0$ if and only if there are a ridge $\rho$ of degree $D$ and a facet
$U\notin\Fset_\rho$ such that $\Fset=\Fset_\rho\cup\{U\}$ and $U$ is adjacent to exactly one
member of $\Fset_\rho$.

Consequently, all connected principal minors of order $D+1$ are nonnegative if and only if every
ridge of degree $D$ is clean.
\end{theorem}

\begin{proof}
If $U$ is adjacent to exactly one member of a maximum-degree fan, Lemma~\ref{lem:borderedfan} gives
$\Phi(\Fset_\rho\cup\{U\})=-D^{D-2}<0$.

Conversely, suppose $|\Fset|=D+1$ and $\Phi(\Fset)<0$. Every proper principal submatrix of
$Q[\Fset]$ has order at most $D$ and is therefore positive semidefinite by
Corollary~\ref{cor:small}; in particular $Q[\Fset]$ itself is not positive semidefinite, and
Lemma~\ref{lem:universal} provides a universal facet $F_*$. Adopt the groups $\Aset_j$ and the
sizes $r_j$ of~\eqref{eq:groupsizes}. If $D=2$, the relations $\sum_j r_j=2$ and $r_j\le1$ force
$r_j=1=D-1$ for some $j$. If $D\ge3$ and every $r_j\le D-2$, then Lemma~\ref{lem:quadform} would
give $Q[\Fset]\succeq0$, a contradiction. Thus in all cases some group has size $D-1$.

Say $r_0=D-1$. Then $F_*$ together with the members of $\Aset_0$ are $D$ distinct facets
containing $\rho_0$, so $d(\rho_0)=D$ and $\{F_*\}\cup\Aset_0=\Fset_{\rho_0}$ is a
maximum-degree fan. Exactly one facet $U$ of $\Fset$ remains. Since $F_*$ is universal, $U$ is
adjacent to $F_*$, and by Lemma~\ref{lem:atmosttwo} it is adjacent to at most two members of the
fan; so it has one or two fan neighbours. If it had two, Lemma~\ref{lem:borderedfan} would give
$\Phi(\Fset)=0$, contradicting $\Phi(\Fset)<0$. Hence $U$ has exactly one fan neighbour.

The final equivalence follows: if some ridge of degree $D$ is not clean, the first paragraph
produces a connected collection of order $D+1$ with negative determinant, and conversely a
connected collection of order $D+1$ with negative determinant produces, by the second paragraph,
a facet meeting a maximum-degree fan in exactly one member.
\end{proof}

\subsection{The refined characterization}

\begin{theorem}\label{thm:refined}
Let $K$ be a finite $k$-dimensional simplicial complex with $k\ge1$ and $K_k\ne\varnothing$, and
let $D=d_1(K)$. If $D=1$, then $\lambda_1(\Lup_{k-1}(K))=k+1=D+k$. Assume now that
$D\ge2$. Then
\[
  \lambda_1\bigl(\Lup_{k-1}(K)\bigr)=D+k
\]
if and only if the following two conditions hold.
\begin{enumerate}[label=\textup{(\roman*)},leftmargin=2.6em,itemsep=0.3em]
\item\label{it:ref-i} Every ridge of degree $D$ is clean.
\item\label{it:ref-ii} For every connected collection $\Fset\subseteq K_k$ satisfying
\[
  |\Fset|\ge D+2\quad\text{and}\quad d_{\Fset}\ge D ,
\]
one has
\[
  \sum_R {(-1)}^{m(R)+c(R)}\,2^{c(R)}\,\eps(R)\,{(D-1)}^{i(R)}\ \ge\ 0,
\]
where the sum ranges over all spanning elementary subgraphs $R$ of $H_{\Fset}$.
\end{enumerate}
\end{theorem}

\begin{proof}
If $D=1$ then no two facets share a ridge, so $S=0$ and~\eqref{eq:lamS} gives
$\lambda_1(\Lup_{k-1})=k+1$. Assume $D\ge2$.

Suppose first that $\lambda_1(\Lup_{k-1}(K))=D+k$. Proposition~\ref{prop:reform} gives
$Q\succeq0$. Every principal submatrix of a positive semidefinite matrix is positive
semidefinite, so $\Phi(\Fset)=\det Q[\Fset]\ge0$ for every $\Fset\subseteq K_k$. For each
collection occurring in~\ref{it:ref-ii}, Proposition~\ref{prop:sachs} identifies this principal
minor with
\[
  \sum_R {(-1)}^{m(R)+c(R)}\,2^{c(R)}\,\eps(R)\,{(D-1)}^{i(R)},
\]
and hence the displayed inequality in~\ref{it:ref-ii} holds. Moreover,
every maximum-degree fan vector has Rayleigh quotient $D-1$, and $D-1$ is now the largest
eigenvalue of $S$, so the fan vector attains the maximum Rayleigh quotient and is an eigenvector;
Proposition~\ref{prop:fanvector} then gives~\ref{it:ref-i}.

Conversely, assume~\ref{it:ref-i} and~\ref{it:ref-ii}; we prove $Q\succeq0$, which for a real
symmetric matrix is equivalent to the nonnegativity of every principal minor. Let
$\Fset\subseteq K_k$ be arbitrary and nonempty.

If $\Gamma[\Fset]$ is disconnected, then as in the proof of Theorem~\ref{thm:allminors} the
determinant factors over the connected components, so it suffices to treat connected $\Fset$.

Let $\Fset$ be connected. If $d_{\Fset}\le D-1$, then Lemma~\ref{lem:gersh} gives the stronger
conclusion $Q[\Fset]\succeq0$, whatever the size of $\Fset$. Suppose then $d_{\Fset}\ge D$. This
already forces $|\Fset|\ge D+1$, since a graph on at most $D$ vertices has maximum degree at most
$D-1$. If $|\Fset|=D+1$, then by~\ref{it:ref-i} and Theorem~\ref{thm:Dplus1} we have
$\Phi(\Fset)\ge0$. The only remaining possibility is $|\Fset|\ge D+2$ together with
$d_{\Fset}\ge D$. For such a collection, condition~\ref{it:ref-ii} says that its signed Sachs
sum is nonnegative, while Proposition~\ref{prop:sachs} identifies this sum with
$\Phi(\Fset)$. Thus $\Phi(\Fset)\ge0$.

Every principal minor of $Q$ has now been shown to be nonnegative, so $Q\succeq0$ and
Proposition~\ref{prop:reform} gives $\lambda_1(\Lup_{k-1}(K))=D+k$.
\end{proof}

\begin{remark}\label{rem:improved}
Theorem~\ref{thm:refined} greatly reduces the determinant test of
Theorem~\ref{thm:allminors}. Gershgorin's theorem handles collections with
$d_{\Fset}\le D-1$, while cleanliness handles collections of exactly $D+1$ facets. Thus
determinants need only be checked when $|\Fset|\ge D+2$ and $d_{\Fset}\ge D$.
Example~\ref{ex:dim3} shows that this remaining condition is necessary.
\end{remark}

\subsection{The cases \texorpdfstring{$D=2$ and $D=3$}{D=2 and D=3}}

The signed-graph reformulation connects the equality problem to the theory of signed graphs with
bounded least eigenvalue. Indeed, by Proposition~\ref{prop:reform},
\begin{equation}\label{eq:lmin}
  \lambda_1\bigl(\Lup_{k-1}(K)\bigr)=D+k
  \quad\Longleftrightarrow\quad
  \lambda_1(S)\le D-1
  \quad\Longleftrightarrow\quad
  \lambda_{\min}(-S)\ge-(D-1).
\end{equation}
The right-hand property is closed under taking induced signed subgraphs, that is under passing to
subcollections of facets, and depends only on the switching class of $\Sigma(K)$. Both features
are invisible in the determinant formulation, and both are exploited below.

\begin{theorem}\label{thm:D2}
Let $K$ be a finite $k$-dimensional simplicial complex, not assumed pure, with $D=2$. The
following are equivalent.
\begin{enumerate}[label=\textup{(\alph*)},leftmargin=2.6em]
\item\label{it:D2a} $\lambda_1(\Lup_{k-1}(K))=k+2$;
\item\label{it:D2b} every connected component of $\Gamma(K)$ is a complete graph;
\item\label{it:D2c} $K_k$ can be partitioned into classes, no ridge being shared between facets
of different classes, such that each class is either a singleton or consists of $m$ facets, with
$2\le m\le k+2$, of one abstract $(k+1)$-simplex.
\end{enumerate}
\end{theorem}

\begin{proof}
Assume~\ref{it:D2a}. Proposition~\ref{prop:reform} gives
$Q=I-S\succeq0$, so $Q$ is the Gram matrix of vectors ${(u_F)}_{F\in K_k}$ with $\|u_F\|^2=1$ and,
for $F\ne G$, with $\langle u_F,u_G\rangle=-\eps(F,G)$ when $F$ and $G$ are adjacent and
$\langle u_F,u_G\rangle=0$ otherwise. If $F$ and $G$ are adjacent then
$|\langle u_F,u_G\rangle|=1=\|u_F\|\|u_G\|$, so equality holds in the Cauchy--Schwarz inequality
and $u_G=\pm u_F$. Conversely, if $u_G=\pm u_F$ with $F\ne G$, their inner product is nonzero, so
$F$ and $G$ are adjacent. Adjacency therefore coincides with the equivalence relation
``$u_F=\pm u_G$'', and every component of $\Gamma(K)$ is complete.

Conversely, assume~\ref{it:D2b}, and let $C$ be a component, necessarily a clique. If $|C|=1$
then $Q[C]=[1]\succeq0$. Suppose $|C|\ge2$. Since $D=2$, no ridge lies in three facets, so three
pairwise adjacent facets cannot share a common ridge; by Lemma~\ref{lem:triangle} every triangle
of $\Gamma[C]$ therefore consists of three facets of one abstract $(k+1)$-simplex, and
Lemma~\ref{lem:sign} gives it the sign $-1$. Thus every triangle of the signed complete graph
$-S[C]$ is positive. Fix $F_0\in C$, set $e_{F_0}=1$, and for $F\ne F_0$ set
$e_F=-S_{F_0,F}$. If $F,G\ne F_0$ are distinct, the positivity of the triangle
$F_0FGF_0$ gives $-S_{F,G}=e_Fe_G$; the same identity holds when one of $F,G$ equals $F_0$
by the definition of the $e_F$. Hence, for $E=\operatorname{diag}(e_F:F\in C)$,
\[
  S[C]=-E(J-I)E .
\]
Consequently
\[
  Q[C]=I-S[C]=I+E(J-I)E=E\,J\,E\ \succeq\ 0 ,
\]
the middle equality because $EIE=I$. Taking the direct sum over components gives $Q\succeq0$,
and Proposition~\ref{prop:reform} gives~\ref{it:D2a}.

It remains to prove the equivalence of~\ref{it:D2b} and~\ref{it:D2c}. Let $C$ be a complete
component with $|C|\ge2$ and
choose distinct $F_1,F_2\in C$; their union $\Omega=F_1\cup F_2$ has $k+2$ vertices. If
$|C|\ge3$ and $F\in C\setminus\{F_1,F_2\}$, then $F_1,F_2,F$ are pairwise adjacent and cannot
share one ridge because $D=2$, so Lemma~\ref{lem:triangle} gives $F_1\cup F_2\cup F=\Omega$ and
hence $F\subseteq\Omega$. Thus every member of $C$ is a facet of the single abstract
$(k+1)$-simplex $\Omega$, and $|C|\le k+2$. Conversely, any two facets of a $(k+1)$-simplex meet
in $k$ vertices, so a class as in~\ref{it:D2c} induces a clique in $\Gamma(K)$, and the
requirement that facets in different classes share no ridge makes each class a component.
\end{proof}

\begin{remark}
For $k=1$,
Theorem~\ref{thm:D2} says that a graph $G$ with $\Delta(G)=2$ has Laplacian spectral radius three
if and only if every component containing an edge is isomorphic to $K_2$, $P_3$, or $K_3$.
These are precisely the connected graphs of maximum degree at most two that possess a dominating
vertex, in agreement with Proposition~\ref{prop:dominating} below.
\end{remark}

For $D=3$, condition~\eqref{eq:lmin} reads $\lambda_{\min}(-S)\ge-2$, which is the hypothesis of
the root-system theory of Cameron, Goethals, Seidel and Shult~\cite{CGSS}: a signed graph has
least eigenvalue at least $-2$ if and only if it is representable in the root system $D_n$ for
some $n$, or in $E_8$. Since the property is closed under induced signed subgraphs, it is
described by its minimal forbidden induced signed subgraphs. Vijayakumar~\cite{Vijayakumar}
proved that every signed graph with least eigenvalue below $-2$ contains an induced signed
subgraph on at most ten vertices whose least eigenvalue is below $-2$. Equivalently, every
minimal forbidden induced signed graph for the condition $\lambda_{\min}\ge-2$ has at most ten
vertices. This bounds the remaining test.

\begin{corollary}\label{cor:D3}
Assume $D=3$. Then $\lambda_1(\Lup_{k-1}(K))=k+3$ if and only if every ridge of degree three
is clean and
\[
  \Phi(\Fset)\ge0\qquad\text{for every connected }\Fset\subseteq K_k
  \text{ with } 5\le|\Fset|\le10 \text{ and } d_{\Fset}\ge3 .
\]
\end{corollary}

\begin{proof}
Necessity is immediate from Theorem~\ref{thm:refined}. Conversely, suppose the stated conditions
hold but equality fails. Then by~\eqref{eq:lmin} the signed graph $-\Sigma(K)$ has least
eigenvalue below $-2$. Choose an induced signed subgraph, on a facet collection $\Fset$, that is
minimal with this property. It is connected, since otherwise one of its components would already
have least eigenvalue below $-2$, and it has at most ten vertices by Vijayakumar's theorem.
Collections with $|\Fset|\le3$ are harmless by Corollary~\ref{cor:small}; collections with
$|\Fset|=4$ are harmless by cleanliness and Theorem~\ref{thm:Dplus1}; and collections with
$d_{\Fset}\le2$ are harmless by Lemma~\ref{lem:gersh}. Hence $\Fset$ satisfies
$5\le|\Fset|\le10$ and $d_{\Fset}\ge3$, so the hypothesis gives $\Phi(\Fset)\ge0$. On the other
hand, minimality makes every proper principal submatrix of $Q[\Fset]$ positive semidefinite,
whereas $Q[\Fset]$ itself is not positive semidefinite; the principal-minor criterion therefore
forces $\Phi(\Fset)=\det Q[\Fset]<0$, a contradiction.
\end{proof}

\subsection{Examples and structural consequences}

Let $k=1$ and let $G$ be a graph, so that the ridges are the vertices, the facets are the edges,
and $\Lup_0(G)=L(G)$.

\begin{proposition}\label{prop:dominating}
Let $v$ be a vertex of maximum degree $D$ in a connected graph $G$. Then $v$ is clean if and only
if $v$ is dominating. Consequently $\lambda_1(L(G))=D+1$ if and only if $G$ has a dominating
vertex.
\end{proposition}

\begin{proof}
The fan of $v$ is the set of edges incident with $v$. If $v$ is dominating, every edge $xy$
outside the fan has both endpoints in $N(v)$, so it is adjacent in $\Gamma(G)$ to exactly the two
fan edges $vx$ and $vy$; hence $v$ is clean.

Conversely, suppose $v$ is not dominating. Choose a shortest path from $v$ to a vertex outside
the closed neighbourhood of $v$; its first three vertices are $v,x,y$ with $vx,xy\in E(G)$ and
$vy\notin E(G)$. The edge $xy$ lies outside the fan and is adjacent to exactly one fan edge,
namely $vx$. Hence $v$ is not clean.

If $\lambda_1(L(G))=D+1$, Proposition~\ref{prop:fanvector} makes every maximum-degree vertex
clean, hence dominating. Conversely, if $v$ is dominating then $D=|V(G)|-1$; the complement of
$G$ has an isolated vertex and is therefore disconnected, so $|V(G)|$ is a Laplacian eigenvalue
of $G$, while every graph of order $n$ satisfies $\lambda_1(L(G))\le n$. Hence
$\lambda_1(L(G))=|V(G)|=D+1$.
\end{proof}

Connectedness is essential for the equivalence between cleanliness and domination: in a
disconnected graph an edge lying in another component meets zero edges of the fan, so it does not
violate cleanliness.

\begin{example}\label{ex:dim3}
Consider the three-dimensional complex generated by
\[
  0123,\ 0125,\ 0136,\ 0235,\ 0245,\ 0345,\ 1356,\ 2346,\ 3456 .
\]
Its maximum upper degree is $D=3$, and the unique ridge of degree three is $025$, whose fan is
$\{0125,0235,0245\}$. Every outside tetrahedron meets this fan in zero or two members, so the
ridge is clean and condition~\ref{it:ref-i} of Theorem~\ref{thm:refined} holds. Nevertheless
$\lambda_1(S)=2.106477882\ldots>2$, and hence
\[
  \lambda_1\bigl(\Lup_2(K)\bigr)=6.106477882\ldots\ >\ 6=D+k .
\]
The obstruction can be seen on the connected collection
\[
  \Fset=\{0123,0136,0235,0345,1356,2346,3456\}.
\]
For this collection,
\[
  Q[\Fset]=
  \begin{pmatrix}
    2&1&-1&0&0&0&0\\
    1&2&0&0&-1&0&0\\
    -1&0&2&1&0&0&0\\
    0&0&1&2&0&0&1\\
    0&-1&0&0&2&0&1\\
    0&0&0&0&0&2&-1\\
    0&0&0&1&1&-1&2
  \end{pmatrix},
  \qquad \det Q[\Fset]=-6 .
\]
Moreover $d_{\Fset}=3=D$, so this collection is not excluded by
Lemma~\ref{lem:gersh}. Thus condition~\ref{it:ref-ii} of Theorem~\ref{thm:refined} cannot be
omitted.
\end{example}

A signed graph is \emph{balanced} if every cycle has positive sign product, equivalently if it is
switching-equivalent to the ordinary adjacency matrix of its underlying graph~\cite{Zaslavsky}.

\begin{theorem}\label{thm:balanced}
Assume $\Gamma(K)$ is connected and $\Sigma(K)$ is balanced. Then
\[
  \lambda_1\bigl(\Lup_{k-1}(K)\bigr)=D+k
\]
if and only if all $k$-faces of $K$ contain one common ridge.
\end{theorem}

\begin{proof}
Since $\Sigma(K)$ is balanced, $S$ is switching-equivalent to $A(\Gamma(K))$ and hence has the
same eigenvalues. By Proposition~\ref{prop:reform}, the equality
$\lambda_1(\Lup_{k-1}(K))=D+k$ holds exactly when
$\lambda_1(A(\Gamma(K)))=D-1$. We now characterize when this occurs.

Suppose first that $\lambda_1(\Lup_{k-1}(K))=D+k$. Choose a ridge of degree $D$.
The $D$ facets containing it form a clique $K_D$ in $\Gamma(K)$. If there were any facet outside
this clique, then, since $\Gamma(K)$ is connected, strict Perron--Frobenius monotonicity would
give
\[
  \lambda_1(A(\Gamma(K)))>\lambda_1(A(K_D))=D-1,
\]
contrary to Proposition~\ref{prop:reform}. Hence these $D$ facets are all the facets of $K$, so
they contain a common ridge.

Conversely, suppose that all facets contain a common ridge. This ridge has degree $D$, and
$\Gamma(K)=K_D$. Thus $S$ is switching-equivalent to $A(K_D)$ and
$\lambda_1(S)=D-1$. Proposition~\ref{prop:reform} now gives
$\lambda_1(\Lup_{k-1}(K))=D+k$.
\end{proof}

\begin{remark}
For $k=1$, balancedness of $\Sigma(K)$ is equivalent to bipartiteness of the graph $G=K$.
Indeed, orienting every edge from one bipartition class to the other makes all incidence signs
at a fixed vertex equal; conversely, such an orientation determines a bipartition according to
the common incidence sign at each vertex. Moreover, all edges contain one common vertex exactly
when a connected bipartite graph is a star, or equivalently when it has a dominating vertex.
Thus Theorem~\ref{thm:balanced} recovers the graph characterization
$\lambda_1(L(G))=\Delta(G)+1$ if and only if $G$ is a star.
\end{remark}

\begin{definition}[Disorientable complex]\label{def:disorientable}
Following~\cite{EidiMukherjee}, we call $K$ \emph{disorientable in dimension $k$} if the $k$-simplices can
be oriented so that any two adjacent $k$-simplices induce the same incidence sign on their common
ridge. With the present sign convention this means that after switching every edge of $\Sigma(K)$
has sign $+1$, that is, $\Sigma(K)$ is balanced.
\end{definition}

\begin{corollary}\label{cor:disorientable}
Suppose $\Gamma(K)$ is connected and $K$ is disorientable in dimension $k$. Then
$\lambda_1(\Lup_{k-1}(K))=D+k$ if and only if all $k$-faces of $K$ contain one common ridge.
\end{corollary}

\begin{corollary}\label{cor:negcycle}
If $\lambda_1(\Lup_{k-1}(K))=D+k$, the graph $\Gamma(K)$ is connected, and $|K_k|>D$, then
$\Sigma(K)$ contains a negative cycle.
\end{corollary}

\begin{proof}
If every cycle were positive the signed graph would be balanced, and
Theorem~\ref{thm:balanced} would force all facets to share one ridge, giving exactly $D$
facets.
\end{proof}

Thus nontrivial extremal complexes are irreducibly signed: negative cycles, of which the negative
triangles supplied by Lemma~\ref{lem:triangle} are the simplest instances, lower the largest
eigenvalue of the signed adjacency matrix below the spectral radius of the underlying unsigned
graph.

\section{The second largest up-Laplacian eigenvalue}\label{sec:second}

We now prove the lower bound. All objects that are specific to the proof, including the two
selected facet stars, are introduced below at the point where they are needed.

\begin{theorem}\label{thm:main}
Let $K$ be a finite $k$-dimensional simplicial complex, where $k\ge1$, with at least two
$k$-faces, and let $\Lup_{k-1}(K)$ be its unnormalized $(k-1)$-dimensional up-Laplacian. Then
\[
  \lambda_2\bigl(\Lup_{k-1}(K)\bigr)\ \ge\ d_2(K)+k-1 .
\]
\end{theorem}

\begin{proof}
Choose distinct $(k-1)$-faces $\sigma$ and $\tau$ such that
\[
  d(\sigma)=d_1(K),\qquad d(\tau)=d_2(K),
\]
and put $d:=d_2(K)$. Such a choice is possible because $K_{k-1}$ has at least two elements, so
that $d_1(K)$ and $d_2(K)$ are the two largest members of a list of at least two upper degrees.
We first assume that $d\ge2$. The case $d=1$ will be treated at the end. 

\medskip
\noindent

Retain \emph{all} $d$ facets containing $\tau$, and retain exactly $d$ facets containing
$\sigma$; the latter is possible because $d(\sigma)=d_1(K)\ge d_2(K)=d$. If $\sigma$ and $\tau$
lie in a common $k$-face, include that common face among the retained facets containing
$\sigma$. Let $H$ be the subcomplex consisting of the entire $(k-1)$-skeleton of $K$, together
with the retained $k$-faces. Define
\[
  \Sset=\{F\in H_k:\ \sigma\subset F\},\qquad \Tset=\{F\in H_k:\ \tau\subset F\} .
\]
Thus $\Sset$ and $\Tset$ are the two selected facet stars, and
\[
  |\Sset|=|\Tset|=d .
\]
Because two distinct $(k-1)$-faces can lie together in at most one $k$-face, a facet lying in
both stars would contain $\sigma\cup\tau$ and hence be uniquely determined; therefore
\[
  |\Sset\cap\Tset|\le1 .
\]
After the reduction, the two possible configurations are exactly
\[
  \Sset\cap\Tset=\varnothing \qquad\text{or}\qquad \Sset\cap\Tset=\{F_0\},
\]
where $F_0$ is the common $k$-face in $\Sset$ and $\Tset$. Since $H$ has the same $(k-1)$-faces
as $K$, Lemma~\ref{lem:deletion} applies and gives
\[
  \lambda_2\bigl(\Lup_{k-1}(K)\bigr)\ \ge\ \lambda_2\bigl(\Lup_{k-1}(H)\bigr) .
\]
It is therefore enough to establish the required estimate for $H$.

\medskip
\noindent

Consider the down Gram matrix
\[
  M={B_k(H)}^{\T}B_k(H) .
\]
 $M$ is
indexed by the $2d$ or $2d-1$ retained facets, and its diagonal is constant. Indeed every
diagonal entry of $M$ is $k+1$, so we may write
\[
  M=kI+N ,
\]
where $N$ has unit diagonal and $N_{F,G}=\eps(F,G)$ for adjacent $F\ne G$, and $N_{F,G}=0$
otherwise. Every eigenvalue of $M$ exceeds the corresponding eigenvalue of $N$ by exactly $k$,
so it suffices to bound $\lambda_2(N)$ from below by $d-1$.

Since $|\Sset\cap\Tset|\le1$, Lemma~\ref{lem:switch} applies to the pair of stars
$\Sset,\Tset$, associated with the $(k-1)$-faces $\sigma,\tau$. We may therefore replace $M$ by
$EME$ and assume from now on that
\[
  N[\Sset]=J_d,\qquad N[\Tset]=J_d .
\]
The switching does not change the spectrum, and by Lemma~\ref{lem:switch} the switched matrix
is again a down-Laplacian matrix of $H$; consequently the incidence
products $\eps$ computed in the new orientations still satisfy Lemma~\ref{lem:sign}, which we
shall use repeatedly. We now treat the two possible intersections of the selected stars
separately.

\medskip
\noindent\textbf{Case 1. $\Sset\cap\Tset=\varnothing$.}
Order the retained facets with those in $\Sset$ first and those in $\Tset$ second. Then
\[
  N=
  \begin{pmatrix}
    J_d & C\\
    C^{\T} & J_d
  \end{pmatrix},
  \qquad
  C=N[\Sset,\Tset] .
\]
Define
\[
  x=\begin{pmatrix}\ones_d\\\mathbf{0}_d\end{pmatrix},
  \qquad
  y=\begin{pmatrix}\mathbf{0}_d\\\ones_d\end{pmatrix},
  \qquad
  s=\ones^{\T}C\ones
   =\sum_{F\in\Sset}\sum_{G\in\Tset}N_{F,G} .
\]
Thus $x$ and $y$ are the $0$--$1$ vectors indicating membership in $\Sset$ and $\Tset$,
respectively. Since the stars are disjoint, every pair $(F,G)\in\Sset\times\Tset$ has
$F\ne G$, so that $C_{F,G}=\eps(F,G)$ when $F$ and $G$ are adjacent and $C_{F,G}=0$ otherwise.

We first exclude one exceptional case.

\begin{claim}\label{cl:pivot}
If $\Sset\cap\Tset=\varnothing$, then $\sigma\cup\tau$ is not a $k$-face of $K$.
\end{claim}

\begin{claimproof}
If $|\sigma\cup\tau|\ne k+1$ there is nothing to prove, since a $k$-face has exactly $k+1$
vertices. So suppose $|\sigma\cup\tau|=k+1$ and that $F:=\sigma\cup\tau$ is a $k$-face of
$K$. Then $\tau\subset F$, and the definition of $H$ retains every $k$-face containing $\tau$;
hence $F\in\Tset$. Also $\sigma\subset F$, so $F$ is a common $k$-face of $\sigma$ and $\tau$.
By the definition of $H$, this common face is also retained among the facets containing
$\sigma$; hence $F\in\Sset$. This contradicts
$\Sset\cap\Tset=\varnothing$.
\end{claimproof}

\begin{claim}\label{cl:s}
In the disjoint-star case, $|s|\le d$.
\end{claim}

\begin{claimproof}
Put
\[
  q=|\sigma\setminus\tau|=|\tau\setminus\sigma| .
\]
Since $\sigma\ne\tau$ and the two faces have the same dimension, $q\ge1$. We divide the
argument according to the value of $q$; the three subcases $q\ge3$, $q=1$ and $q=2$ are
exhaustive and mutually exclusive.

\smallskip
\noindent\emph{Subcase 1.1: $q\ge3$.} Every facet in $\Sset$ is obtained from $\sigma$ by
adding one vertex, and every facet in $\Tset$ is obtained from $\tau$ by adding one vertex.
Since $\sigma$ and $\tau$ differ in at least three vertices, Lemma~\ref{lem:adj}\ref{it:adj-a}
shows that no such pair can share a ridge. Hence $C$ is the zero matrix, and therefore $s=0$.

\smallskip
\noindent\emph{Subcase 1.2: $q=1$.} Write
\[
  \sigma=\rho\cup\{a\},\qquad \tau=\rho\cup\{b\},\qquad |\rho|=k-1 .
\]
Every facet in $\Sset$ has the form $F_u=\rho\cup\{a,u\}$ with $u\notin\sigma$, and every facet
in $\Tset$ has the form $G_v=\rho\cup\{b,v\}$ with $v\notin\tau$. We first check that the
degenerate values $u=b$ and $v=a$ do not occur. If $u=b$ then $F_u=\rho\cup\{a,b\}=\sigma\cup\tau$
would be a $k$-face of $K$, contradicting Claim~\ref{cl:pivot}; the same argument excludes
$v=a$.

With $u\ne b$ and $v\ne a$ available, Lemma~\ref{lem:adj}\ref{it:adj-b} shows that $F_u$ and
$G_v$ share a ridge precisely when $u=v$. Thus every row and every column of $C$ contains at
most one nonzero entry. Consequently $C$ has at most $d$ nonzero entries, each equal to $1$ or
$-1$, and hence
\[
  |s|\le d .
\]

\smallskip
\noindent\emph{Subcase 1.3: $q=2$.} Write
\[
  \sigma=\rho\cup\{a_1,a_2\},\qquad \tau=\rho\cup\{b_1,b_2\},
\]
where the four displayed vertices are distinct. By Lemma~\ref{lem:adj}\ref{it:adj-c} the only
facets that can participate in a cross-interaction are
\[
  F_i=\sigma\cup\{b_i\},\qquad G_j=\tau\cup\{a_j\},\qquad i,j\in\{1,2\},
\]
These four are $(k+1)$-element subsets of the same $(k+2)$-element vertex set
$W=\sigma\cup\tau$, which need not be a face of $K$, and for every $i,j$,
\[
  F_i\cap G_j=\rho\cup\{b_i,a_j\},
\]
which is a ridge. Thus every retained $F_i$ interacts with every retained $G_j$.

Assume that both $G_1$ and $G_2$ are retained. For each retained $F_i$, the three facets
$F_i,G_1,G_2$ are distinct $(k+1)$-element subsets of the same vertex set $W$. By the chosen
diagonal switching, the $\Tset$-block of $N$ is positive. Since $G_1,G_2\in\Tset$, we have
\[
  \eps(G_1,G_2)=N_{G_1,G_2}=1 .
\]
Lemma~\ref{lem:sign} applied to the triple $F_i,G_1,G_2$ gives
$\eps(F_i,G_1)\eps(G_1,G_2)\eps(G_2,F_i)=-1$, whence
\[
  \eps(F_i,G_1)=-\eps(F_i,G_2) .
\]
Thus, for every retained $F_i$, the two entries corresponding to $G_1$ and $G_2$ have opposite
signs and hence sum to zero. Similarly, if both $F_1$ and $F_2$ are retained, the two entries
in each retained column sum to zero, because $F_1,F_2\in\Sset$ gives
$\eps(F_1,F_2)=1$ and Lemma~\ref{lem:sign} applied to $F_1,F_2,G_j$ gives
$\eps(F_1,G_j)=-\eps(F_2,G_j)$.

It remains to combine these observations. If neither $F_1$ nor $F_2$ is retained, or neither
$G_1$ nor $G_2$ is retained, then the cross block is zero and $s=0$. If both $G_1,G_2$ are
retained, its rows cancel; if both $F_1,F_2$ are retained, its columns cancel. Otherwise, at
most one $F_i$ and at most one $G_j$ are retained, so there is at most one nonzero cross-entry
and $|s|\le1$. In every configuration,
\[
  |s|\le1\le d .
\]
This completes all three subcases and proves the claim.
\end{claimproof}

We now use Claim~\ref{cl:s} to obtain the required eigenvalue estimate. From the definitions
of $x$ and $y$,
\[
  x^{\T}x=y^{\T}y=d,\qquad x^{\T}y=0 .
\]
The block form of $N$ gives
\[
  x^{\T}Nx=y^{\T}Ny=d^2,
  \qquad
  x^{\T}Ny=s .
\]
Set
\[
  Q=\begin{pmatrix}x/\sqrt d & y/\sqrt d\end{pmatrix} .
\]
The columns of $Q$ are orthonormal, so $Q^{\T}Q=I_2$. The preceding identities give
\[
  Q^{\T}NQ
  =\begin{pmatrix}
      d & s/d\\[2pt]
      s/d & d
    \end{pmatrix} .
\]
The smaller eigenvalue of this $2\times2$ matrix is
\[
  d-\frac{|s|}{d}\ \ge\ d-1
\]
by Claim~\ref{cl:s}. Applying Lemma~\ref{lem:compress} gives
\[
  \lambda_2(N)\ \ge\ d-1 ,
\]
and hence, since $M=kI+N$,
\[
  \lambda_2(M)\ \ge\ k+d-1 .
\]

\medskip
\noindent\textbf{Case 2. $\Sset\cap\Tset=\{F_0\}$.}
Because $F_0$ contains both $\sigma$ and $\tau$, we have $|\sigma\cup\tau|\le k+1$, so these two
$(k-1)$-faces differ in exactly one vertex, that is $q=1$. Write
\[
  \rho=\sigma\cap\tau,\qquad \sigma=\rho\cup\{a\},\qquad \tau=\rho\cup\{b\} .
\]
Then $\sigma\cup\tau$ has $k+1$ vertices and is contained in $F_0$, so
\[
  F_0=\rho\cup\{a,b\} .
\]
Every other facet in $\Sset$ has the form $F_u=\rho\cup\{a,u\}$ with $u\notin\sigma$ and, since
$F_u\ne F_0$, with $u\ne b$; every other facet in $\Tset$ has the form $G_v=\rho\cup\{b,v\}$
with $v\notin\tau$ and $v\ne a$. The next claim describes all entries of $N$ between
$\Sset\setminus\{F_0\}$ and $\Tset\setminus\{F_0\}$ and gives the resulting block form of $N$.

\begin{claim}\label{cl:structure}
The cross-interactions between $\Sset\setminus\{F_0\}$ and $\Tset\setminus\{F_0\}$ form a
partial matching, and every matched cross-entry is $-1$. Consequently, with the ordering
\[
  F_0,\quad \Sset\setminus\{F_0\},\quad \Tset\setminus\{F_0\},
\]
the matrix $N$ has the form
\[
  N=\begin{pmatrix}
    1 & \ones^{\T} & \ones^{\T}\\
    \ones & J_{d-1} & -P\\
    \ones & -P^{\T} & J_{d-1}
  \end{pmatrix},
\]
where $P$ is a $0,1$ partial permutation matrix.
\end{claim}

\begin{claimproof}
Since $u\ne b$ and $v\ne a$, Lemma~\ref{lem:adj}\ref{it:adj-b} applies and shows that $F_u$ and
$G_v$ share a ridge if and only if $u=v$. Concretely, when $u=v$ the intersection
$F_u\cap G_v=\rho\cup\{u\}$ has $k$ vertices, whereas for $u\ne v$ their intersection is $\rho$,
which has only $k-1$ vertices. Hence the cross-interactions form a partial matching, and the
corresponding block has at most one nonzero entry in each row and each column, every such entry
being either $1$ or $-1$. It remains to determine their signs.

Consider any matched pair $F_u,G_u$. Since $u\notin\rho\cup\{a,b\}$, the set
\[
  W_u=\rho\cup\{a,b,u\}
\]
has $k+2$ vertices, and
\[
  F_0=W_u\setminus\{u\},\qquad
  F_u=W_u\setminus\{b\},\qquad
  G_u=W_u\setminus\{a\} .
\]
Thus $F_0,F_u,G_u$ are three distinct $(k+1)$-element subsets of $W_u$, so
Lemma~\ref{lem:sign} applies to them. Moreover, the switched star blocks are positive. Since
$F_0,F_u\in\Sset$ and $F_0,G_u\in\Tset$, we have
\[
  \eps(F_0,F_u)=N_{F_0,F_u}=1,\qquad \eps(G_u,F_0)=N_{G_u,F_0}=1 .
\]
Lemma~\ref{lem:sign} therefore gives
\[
  \eps(F_0,F_u)\,\eps(F_u,G_u)\,\eps(G_u,F_0)=-1 .
\]
Substituting the two known signs yields $\eps(F_u,G_u)=-1$. Since the matched pair was
arbitrary, every nonzero cross-entry is $-1$. Hence the cross block is $-P$, where $P$ is a
$0$--$1$ partial permutation matrix of order $d-1$.

For the displayed shape of $N$ it remains to identify the entries in the row and column of
$F_0$. If $F\in\Sset\setminus\{F_0\}$ then $F$ and $F_0$ both lie in $\Sset$, so
$N_{F_0,F}=1$; if $G\in\Tset\setminus\{F_0\}$ then $G$ and $F_0$ both lie in $\Tset$, so
$N_{F_0,G}=1$. The diagonal entry is $N_{F_0,F_0}=1$, and the two diagonal blocks are $J_{d-1}$
because $\Sset$ and $\Tset$ have positive blocks. This proves both assertions of the claim.
\end{claimproof}

With the ordering in Claim~\ref{cl:structure}, define
\[
  x=\begin{pmatrix}1\\\ones_{d-1}\\\mathbf{0}_{d-1}\end{pmatrix},
  \qquad
  y=\begin{pmatrix}1\\\mathbf{0}_{d-1}\\\ones_{d-1}\end{pmatrix} .
\]
Thus the coordinate of $x$ indexed by a facet is $1$ exactly when that facet lies in $\Sset$;
the analogous statement holds for $y$ and $\Tset$.

Let $r$ be the number of nonzero entries of $P$. Since $P$ is a partial permutation matrix of
order $d-1$,
\[
  0\le r\le d-1 .
\]
 Note that $x$ and $y$ are
linearly independent: they are $0$--$1$ vectors with $d\ge2$ ones each, and they agree in only
one coordinate.

\begin{claim}\label{cl:rayleigh}
Every nonzero vector in $\spn\{x,y\}$ has Rayleigh quotient with respect to $N$ at least $d-1$.
\end{claim}

\begin{claimproof}
Since the two stars meet only in $F_0$,
\[
  x^{\T}x=y^{\T}y=d,\qquad x^{\T}y=1 .
\]
Since $N[\Sset]=N[\Tset]=J_d$,
\[
  x^{\T}Nx=y^{\T}Ny=d^2 .
\]
Since $P$ is a $0$--$1$ matrix with $r$ nonzero entries,
$\ones_{d-1}^{\T}P\ones_{d-1}=r$. Direct multiplication gives
\[
  x^{\T}Ny=1+(d-1)+(d-1)-r=2d-1-r .
\]
Put
\[
  p=x+y,\qquad w=x-y .
\]
These vectors are orthogonal, because
\[
  p^{\T}w=x^{\T}x-y^{\T}y=0 ,
\]
and they are nonzero, since $p^{\T}p=2(d+1)>0$ and $w^{\T}w=2(d-1)>0$, the latter using
$d\ge2$. They are also orthogonal with respect to the quadratic form of $N$, since
\[
  p^{\T}Nw=x^{\T}Nx-x^{\T}Ny+y^{\T}Nx-y^{\T}Ny=d^2-(2d-1-r)+(2d-1-r)-d^2=0 .
\]
Since $\{p,w\}$ is a basis of $\spn\{x,y\}$ and both mixed terms vanish, the Rayleigh quotient
of any nonzero vector $z=\alpha p+\beta w$ is a weighted average of the Rayleigh quotients of
$p$ and $w$. It therefore suffices to show that both of these quotients are at least $d-1$.

For $w=x-y$,
\[
  w^{\T}w=2(d-1),\qquad
  w^{\T}Nw=2\bigl(d^2-(2d-1-r)\bigr)=2\bigl({(d-1)}^{2}+r\bigr) .
\]
Therefore
\[
  \frac{w^{\T}Nw}{w^{\T}w}=\frac{{(d-1)}^{2}+r}{d-1}=d-1+\frac{r}{d-1}\ \ge\ d-1 ,
\]
since $r\ge0$ and $d-1>0$.

For $p=x+y$,
\[
  p^{\T}p=2(d+1),\qquad p^{\T}Np=2\bigl(d^2+2d-1-r\bigr) .
\]
Consequently
\[
  \frac{p^{\T}Np}{p^{\T}p}
  =\frac{d^2+2d-1-r}{d+1}
  =d-1+\frac{2d-r}{d+1} >\ d-1 ,
\]
because $r\le d-1<2d$. Both orthogonal directions therefore have Rayleigh quotient at least
$d-1$, and the mixed quadratic term vanishes. The claim follows.
\end{claimproof}

By Claim~\ref{cl:rayleigh} and Lemma~\ref{lem:compress}\ref{it:comp-a}, applied to the
two-dimensional subspace $\spn\{x,y\}$,
\[
  \lambda_2(N)\ \ge\ d-1 .
\]
Hence, also in the one-common-facet case,
\[
  \lambda_2(M)\ \ge\ k+d-1 .
\]

\medskip
\noindent

For $d\ge2$, the preceding two cases give $\lambda_2(M)\ge k+d-1>0$.
Lemma~\ref{lem:gram}\ref{it:gram-b} and deletion monotonicity therefore give
\[
  \lambda_2\bigl(\Lup_{k-1}(K)\bigr)
  \ \ge\ \lambda_2\bigl(\Lup_{k-1}(H)\bigr)
  \ =\ \lambda_2(M)
  \ \ge\ k+d-1
  \ =\ d_2(K)+k-1 .
\]

\medskip
It remains to prove the inequality when $d=1$. Since $K$ has at least two $k$-faces, choose
distinct $F_1,F_2\in K_k$. The
principal submatrix of $M(K)$ indexed by these faces has the form
\[
  \begin{pmatrix}k+1&\eps\\ \eps&k+1\end{pmatrix},
  \qquad \eps\in\{0,1,-1\}.
\]
Its smaller eigenvalue is $k+1-|\eps|\ge k>0$. Hence interlacing and
Lemma~\ref{lem:gram}\ref{it:gram-b} give
\[
  \lambda_2\bigl(\Lup_{k-1}(K)\bigr)
  =\lambda_2\bigl(M(K)\bigr)\ge k=d+k-1 .
\]
The proof is complete.
\end{proof}

\begin{remark}
For $k=1$, the complex $K$ is a graph and $\Lup_0(K)$ is its ordinary Laplacian. Thus
Theorem~\ref{thm:main} becomes
\[
  \lambda_2(L(G))\ \ge\ d_2(G) .
\]
Hence we recover the corresponding inequality for graphs.
\end{remark}

\begin{remark}
The assumption that $K$ contains at least two $k$-faces is necessary. If $K$ consists of a
single $k$-simplex, then $B_k$ has one column, so $\Lup_{k-1}$ has rank one and
$\lambda_2(\Lup_{k-1})=0$. On the other hand, every ridge has upper degree one, and hence
$d_2(K)+k-1=k$.
\end{remark}

\begin{remark}
The bound is attained. For $d_2(K)=1$, let $K$ consist of two $k$-simplices glued along a common
ridge; the two columns of $B_k$ have squared norm $k+1$ and incidence product $\pm1$, so the
nonzero eigenvalues of $\Lup_{k-1}$ are $k+2$ and $k$, while $d_2(K)=1$, whence
$\lambda_2=k=d_2(K)+k-1$. For any $d\ge2$, fix $\rho$ with $|\rho|=k-1$ and distinct vertices
$a,b,u_1,\dots,u_d$ outside $\rho$, and let $K$ be generated by the $2d$ facets
$\rho\cup\{a,u_i\}$ and $\rho\cup\{b,u_i\}$. Then $d_1(K)=d_2(K)=d$, the two stars of
$\sigma=\rho\cup\{a\}$ and $\tau=\rho\cup\{b\}$ are disjoint, and $\sigma\cup\tau$ is not a
facet, so Case~1 with $q=1$ applies and the matching is complete. Put
$F_i=\rho\cup\{a,u_i\}$, $G_i=\rho\cup\{b,u_i\}$, and
$P=\rho\cup\{a,b\}$. After switching the two stars so that their diagonal blocks are $J_d$,
let $e_{F_i},e_{G_i}\in\{\pm1\}$ be the switching multipliers. Then
$e_{F_i}[\sigma:F_i]=c_{\Sset}$ and $e_{G_i}[\tau:G_i]=c_{\Tset}$ are constant within their
respective stars. Applying Lemmas~\ref{lem:triangle} and~\ref{lem:sign} to the abstract
$(k+1)$-simplex $P\cup\{u_i\}$ gives
\[
  \eps(F_i,G_i)
  =-[\sigma:P][\sigma:F_i][\tau:P][\tau:G_i].
\]
Hence every matched cross-entry has the same switched sign
\[
  e_{F_i}e_{G_i}\eps(F_i,G_i)
  =-[\sigma:P][\tau:P]c_{\Sset}c_{\Tset}=:\eps_0.
\]
Thus the cross-block is $\eps_0I_d$. If $\eps_0=1$, switching every coordinate in the
$\Tset$-star preserves both diagonal blocks and changes the cross-block sign. We may therefore
assume that
\[
  N=\begin{pmatrix} J_d & -I_d\\ -I_d & J_d\end{pmatrix},
\]
whose eigenvalues are $d+1$, $d-1$, and $\pm1$ with multiplicity $d-1$ each. Hence
$\lambda_2(N)=d-1$ and $\lambda_2(\Lup_{k-1}(K))=k+d-1=d_2(K)+k-1$.
\end{remark}

\begin{remark}
A natural extension of Theorem~\ref{thm:main} to higher eigenvalue indices would be
\[
  \lambda_m\bigl(\Lup_{k-1}(K)\bigr)\ \ge\ d_m(K)-m+k+1 .
\]
This fails already for $m=3$. Indeed, let $K$ be the cone over a $3$-cycle: its vertices are
$v,1,2,3$, and its facets are
\[
  \{v,1,2\},\qquad \{v,2,3\},\qquad \{v,3,1\}.
\]
The three edges incident with $v$ have upper degree $2$, while the remaining three edges have
upper degree $1$; hence $d_3(K)=2$. For suitable orientations, the down Gram matrix is
\[
  M(K)=
  \begin{pmatrix}
    3&-1&-1\\
    -1&3&-1\\
    -1&-1&3
  \end{pmatrix},
\]
whose eigenvalues are $4,4,1$. Therefore the three nonzero eigenvalues of $\Lup_1(K)$ are also
$4,4,1$. Taking $k=2$ and $m=3$ in the proposed inequality would give
\[
  1=\lambda_3\bigl(\Lup_1(K)\bigr)\ \ge\ d_3(K)-3+2+1=2,
\]
a contradiction. Thus this degree-based extension does not hold in general for higher
eigenvalue indices.
\end{remark}


\begin{thebibliography}{99}

\bibitem{BrouwerHaemers}
Brouwer, A.E., Haemers, W.H.: A lower bound for the Laplacian eigenvalues of a graph---proof
of a conjecture by Guo. Linear Algebra Appl. \textbf{429}, 2131--2135 (2008).
\href{https://doi.org/10.1016/j.laa.2008.06.008}{https://doi.org/10.1016/j.laa.2008.06.008}

\bibitem{CGSS}
Cameron, P.J., Goethals, J.M., Seidel, J.J., Shult, E.E.: Line graphs, root systems, and
elliptic geometry. J. Algebra \textbf{43}, 305--327 (1976).
\href{https://doi.org/10.1016/0021-8693(76)90162-9}{https://doi.org/10.1016/0021-8693(76)90162-9}

\bibitem{DuvalReiner}
Duval, A.M., Reiner, V.: Shifted simplicial complexes are Laplacian integral.
Trans. Amer. Math. Soc. \textbf{354}, 4313--4344 (2002).
\href{https://doi.org/10.1090/S0002-9947-02-03082-9}{https://doi.org/10.1090/S0002-9947-02-03082-9}

\bibitem{EidiMukherjee}
Eidi, M., Mukherjee, S.: Higher order bipartiteness vs bi-partitioning in simplicial complexes.
In: 41st International Symposium on Computational Geometry (SoCG 2025), Leibniz International
Proceedings in Informatics (LIPIcs), vol.~332, Article 45, pp.~45:1--45:12. Schloss
Dagstuhl--Leibniz-Zentrum f\"ur Informatik (2025).
\href{https://doi.org/10.4230/LIPIcs.SoCG.2025.45}{https://doi.org/10.4230/LIPIcs.SoCG.2025.45}

\bibitem{GroneMerris}
Grone, R., Merris, R.: The Laplacian spectrum of a graph II\@.
SIAM J. Discrete Math. \textbf{7}, 221--229 (1994).
\href{https://doi.org/10.1137/S0895480191222653}{https://doi.org/10.1137/S0895480191222653}

\bibitem{Guo}
Guo, J.-M.: On the third largest Laplacian eigenvalue of a graph.
Linear Multilinear Algebra \textbf{55}, 93--102 (2007).
\href{https://doi.org/10.1080/03081080600730996}{https://doi.org/10.1080/03081080600730996}

\bibitem{Huang}
Huang, J.: The Duval--Reiner conjecture: counterexamples and the second partial-sum inequality.
arXiv:2607.20051 (2026).
\href{https://arxiv.org/abs/2607.20051}{https://arxiv.org/abs/2607.20051}



\bibitem{HornJohnson}
Horn, R.A., Johnson, C.R.: Matrix Analysis, 2nd edn. Cambridge University Press,
Cambridge (2013).

\bibitem{LiPan}
Li, J.-S., Pan, Y.-L.: A note on the second largest eigenvalue of the Laplacian matrix of a
graph. Linear Multilinear Algebra \textbf{48}, 117--121 (2000).
\href{https://doi.org/10.1080/03081080008818663}{https://doi.org/10.1080/03081080008818663}

\bibitem{Merris}
Merris, R.: Degree maximal graphs are Laplacian integral.
Linear Algebra Appl. \textbf{199}, 381--389 (1994).
\href{https://doi.org/10.1016/0024-3795(94)90361-1}{https://doi.org/10.1016/0024-3795(94)90361-1}

\bibitem{Vijayakumar}
Vijayakumar, G.R.: Signed graphs represented by $D_\infty$.
European J. Combin. \textbf{8}, 103--112 (1987).
\href{https://doi.org/10.1016/S0195-6698(87)80024-0}{https://doi.org/10.1016/S0195-6698(87)80024-0}

\bibitem{ZhangSongFan}
Zhang, H.-Z., Song, Y.-M., Fan, Y.-Z.: Degree majorization and Laplacian eigenvalue sums for
simplicial complexes. arXiv:2607.20910 (2026).
\href{https://arxiv.org/abs/2607.20910}{https://arxiv.org/abs/2607.20910}

\bibitem{Zaslavsky}
Zaslavsky, T.: Signed graphs. Discrete Appl. Math. \textbf{4}, 47--74 (1982);
erratum, \textbf{5}, 248 (1983).

\end{thebibliography}
\end{document}